\documentclass[a4paper]{amsart}
\usepackage{hyperref}
\usepackage[utf8]{inputenc}
\usepackage{txfonts, amsmath,amstext,amsthm,amscd,amsopn,verbatim,amssymb, amsfonts}
\usepackage{fullpage}
\usepackage{todonotes}
\usepackage{mathtools}

\usepackage[bbgreekl]{mathbbol}
\usepackage{enumerate}

\usepackage{float}
\restylefloat{table}

\usepackage{tikz}
\usepackage{tikz-cd}
\usetikzlibrary{matrix}
\usetikzlibrary{shapes}
\usetikzlibrary{arrows}
\usetikzlibrary{calc,3d}
\usetikzlibrary{decorations,decorations.pathmorphing,decorations.pathreplacing,decorations.markings}
\usetikzlibrary{through}

\tikzset{snakeit/.style={decorate, decoration={snake, amplitude=.2mm,segment length=1mm}}}

\tikzset{ext/.style={circle, draw,inner sep=1pt}, int/.style={circle,draw,fill,inner sep=2pt},nil/.style={inner sep=1pt}}
\tikzset{cy/.style={circle,draw,fill,inner sep=2pt},scy/.style={circle,draw,inner sep=2pt},scyx/.style={draw,cross out,inner sep=2pt},scyt/.style={draw,regular polygon,regular polygon sides=3,inner sep=0.95pt}}
\tikzset{exte/.style={circle, draw,inner sep=3pt},inte/.style={circle,draw,fill,inner sep=3pt}}
\tikzset{diagram/.style={matrix of math nodes, row sep=3em, column sep=2.5em, text height=1.5ex, text depth=0.25ex}}
\tikzset{diagram2/.style={matrix of math nodes, row sep=0.5em, column sep=0.5em, text height=1.5ex, text depth=0.25ex}}
\tikzset{rowcolsep/.style={column sep=.2cm, row sep=.1cm}}
\tikzset{cross/.style={cross out,draw}}
\tikzset{every loop/.style={draw}}

\tikzset{
  crossed/.style={
    decoration={markings,mark=at position .5 with {\arrow{|}}},
    postaction={decorate},
    shorten >=0.4pt}}

\tikzset{every picture/.style={baseline=-.65ex} }

\newcommand{\notadp}{{
\begin{tikzpicture}[baseline=-.55ex,scale=.2, every loop/.style={}]
 \node[circle,draw,fill,inner sep=.5pt] (a) at (0,0) {};
 \draw (a) edge[loop] (a);
 \draw (-.3,.1) -- (.3,.5);
\end{tikzpicture}}}

\newcommand{\tp}{{
\begin{tikzpicture}[baseline=-.55ex,scale=.2, every loop/.style={}]
 \node[circle,draw,fill,inner sep=.5pt] (a) at (0,0) {};
 \draw (a) edge[loop] (a);
\end{tikzpicture}}}

\theoremstyle{plain}
  \newtheorem{thm}{Theorem}
  \newtheorem{defi}[thm]{Definition}
  \newtheorem{prop}[thm]{Proposition}
  
  \newtheorem{cor}[thm]{Corollary}
  
  \newtheorem{lemma}[thm]{Lemma}
\theoremstyle{definition}
  \newtheorem{ex}{Example}
  \newtheorem{rem}{Remark}

\newcommand{\alg}[1]{\mathfrak{{#1}}}

\newcommand{\ad}{{\text{ad}}}

\newcommand{\C}{{\mathbb{C}}}

\newcommand{\K}{{\mathbb{K}}}

\newcommand{\Graphs}{{\mathsf{Graphs}}}
\newcommand{\fGraphs}{{\mathsf{fGraphs}}}

\newcommand{\mF}{\mathcal{F}}

\newcommand{\FreeLie}{\mathrm{FreeLie}}
\newcommand{\FreeCom}{S}

\newcommand{\Lie}{\mathsf{Lie}}

\newcommand{\Ass}{\mathsf{Assoc}}
\newcommand{\Com}{\mathsf{Com}}

\newcommand{\bpm}{\begin{pmatrix}}
\newcommand{\epm}{\end{pmatrix}}

\newcommand{\Aut}{\mathrm{Aut}}
\newcommand{\GC}{\mathsf{GC}}

\newcommand{\HGC}{\mathsf{HGC}}

\newcommand{\G}{\mathsf{G}}

\DeclareMathOperator{\gr}{gr}

\newcommand{\gra}{\mathit{gra}}
\newcommand{\grt}{\alg {grt}}

\newcommand{\MM}{{\mathcal M}}

\newcommand{\Q}{\mathbb{Q}}

\newcommand{\fg}{\mathfrak{g}}
\newcommand{\Sp}{\mathrm{Sp}}
\renewcommand{\Bar}{\mathtt{B}}
\newcommand{\OSp}{\mathrm{OSp}}

\newcommand{\osp}{\mathfrak{osp}}
\renewcommand{\sp}{\mathfrak{sp}}
\newcommand{\ft}{\mathfrak{t}}

\DeclareMathOperator{\tru}{\mathrm{tr}}

\newcommand{\fc}{\mathfrak c}

\newcommand{\POp}{\mathcal P}
\newcommand{\dg}{\mathrm{dg}}
\newcommand{\Alg}{\mathrm{Alg}}
\newcommand{\Free}{\mathbb{F}}
\newcommand{\fG}{\mathsf{fG}}
\newcommand{\GCex}{\GC^\ex}
\newcommand{\Gex}{\G^\ex}
\newcommand{\fGex}{\fG^\ex}
\renewcommand{\ex}{{\mathrm{ex}}}

\newcommand{\krv}{\mathfrak{krv}}
\newcommand{\wgeis}{\FreeLie(\bar H(W_{g,1}))}
\newcommand{\Der}{\mathrm{Der}}
\renewcommand{\div}{\mathrm{div}}
\newcommand{\icggeis}{\icg_{(g),1}}

\newcommand{\icg}{\mathsf{ICG}}
\newcommand{\icgg}{\mathsf{ICG}_{(\mathrm{g})}}
\newcommand{\Graphsg}{\mathsf{Graphs}_{(\mathrm{g})}}

\newcommand{\GCg}{\mathsf{GC}_{(\mathrm{g})}}

\newcommand{\cut}{\mathrm{cut}}

\author{Matteo Felder}
\address{Department of Mathematics\\ ETH Zurich\\ R\"amistrasse 101 \\ 8092 Zurich, Switzerland}
\email{matteo.felder@math.ethz.ch}

\author{Florian Naef}
\address{Department of Mathematical Sciences, University of Copenhagen, Copenhagen, Denmark}
\email{flna@math.ku.dk}

\author{Thomas Willwacher}
\address{Department of Mathematics\\ ETH Zurich\\ R\"amistrasse 101 \\ 8092 Zurich, Switzerland}
\email{thomas.willwacher@math.ethz.ch}

\begin{document}
\title{Stable cohomology of graph complexes}

\thanks{
  M.F. and T.W. have been supported by the ERC starting grant 678156 GRAPHCPX, and the NCCR SwissMAP, funded by the Swiss National Science Foundation. F.N. has been supported by the Marie Sklodowska-Curie grant 896370.
}



\begin{abstract}
We study three graph complexes related to the higher genus Grothendieck-Teichmüller Lie algebra and diffeomorphism groups of manifolds. We show how the cohomology of these graph complexes is related, and we compute the cohomology as the genus $g$ tends to $\infty$. As a byproduct, we find that the Malcev completion of the genus $g$ mapping class group relative to the symplectic group is Koszul in the stable limit (partially answering a question of Hain). Moreover, we obtain that any elliptic associator gives a solution to the elliptic Kashiwara-Vergne problem.
\end{abstract}

\maketitle

\section{Introduction}
Let $m$ be a fixed positive integer.
In this paper we study three related graph complexes that arise in connection to the diffeomorphism groups of the manifolds 
\[
  W_g = \#^g S^{m}\times S^m\quad \text{and} \quad W_{g,1} = W_g \setminus D^{2m}.
\]
More precisely, 
we consider graph complexes $\GC_{(g)}$, $\GC_{(g),1}$ and $\GC_{(g),1}^\tp$ as follows.
\begin{itemize}
\item Elements of $\GC_{(g),1}^\tp$ are $\Q$-linear series of isomorphism classes of connected, at least trivalent graphs whose vertices are decorated by elements of the reduced homology $\bar H_\bullet(W_{g,1})$. As the superscript indicates, these graphs may have tadpoles, that is edges connecting a vertex to itself.
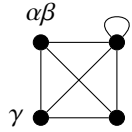
\[
\begin{tikzpicture}[yshift=-.5cm]
\node[int,label=180:{$\gamma$}] (v1) at (0,0) {};
\node[int] (v2) at (1,0) {};
\node[int,label=90:{$\alpha\beta$}] (v3) at (0,1) {};
\node[int] (v4) at (1,1) {};
\draw (v1) edge (v3) edge (v2) edge (v4) 
    (v4) edge[loop] (v4) edge (v2) edge (v3) 
    (v3) edge (v2);
\end{tikzpicture}
\quad\quad \text{with $\alpha,\beta,\gamma \in H_m(W_{g,1})$}
\]
The differential on these complexes has two terms, $\delta=\delta_{split}+\delta_{glue}$. The piece $\delta_{split}$ is defined by summing over vertices, and splitting the vertex.
  \begin{align}\label{equ:deltasplit}
    \delta_{split} \Gamma &= \sum_{v \text{ vertex} }  \pm 
    \Gamma\text{ split $v$} 
    &
    \begin{tikzpicture}[baseline=-.65ex]
    \node[int] (v) at (0,0) {};
    \draw (v) edge +(-.3,-.3)  edge +(-.3,0) edge +(-.3,.3) edge +(.3,-.3)  edge +(.3,0) edge +(.3,.3);
    \end{tikzpicture}
    &\mapsto
    \sum
    \begin{tikzpicture}[baseline=-.65ex]
    \node[int] (v) at (0,0) {};
    \node[int] (w) at (0.5,0) {};
    \draw (v) edge (w) (v) edge +(-.3,-.3)  edge +(-.3,0) edge +(-.3,.3)
     (w) edge +(.3,-.3)  edge +(.3,0) edge +(.3,.3);
    \end{tikzpicture}
  \end{align}
The piece $\delta_{glue}\Gamma$ is defined on a graph $\Gamma$ by summing over all pairs $(\alpha,\beta)$ of $\bar H_\bullet(W_{g,1})$-decorations in the graph $\Gamma$, replacing the pair of decorations by an edge, and multiplying the graph with the numeric prefactor $\langle \alpha,\beta\rangle$, using the canonical pairing $\langle -,-\rangle: \bar H_\bullet(W_{g,1})\times \bar H_\bullet(W_{g,1}) \to \Q$.
\item The complex 
\[
  \GC_{(g),1} = \GC_{(g),1}^\tp / I_g^{\tp}
\]
is the quotient obtained by setting all graphs with tadpoles to zero.
\item The complex $\GC_{(g)}$ is defined similarly to $\GC_{(g),1}$, except for three differences. First, one decorates vertices by $\bar H_\bullet(W_{g})$ instead of $\bar H_\bullet(W_{g,1})$. Second, the piece of the differential $\delta_{glue}$ uses the pairing 
\begin{equation}\label{equ:Hg pairing}
\langle -,-\rangle: H_\bullet(W_{g})\times H_\bullet(W_{g}) \to \Q
\end{equation}
instead.

Third, there is an additional piece of the differential $\delta_Z$ that glues a new vertex to a decoration with the top class $\omega\in H_{2m}(W_{g})$. The new vertex is then decorated by the canonical diagonal element $\Delta_1\in H_m(W_{g})\otimes H_m(W_{g})$.
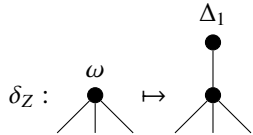
\begin{align*}
\delta_Z :
\begin{tikzpicture}
\node[int,label=90:{$\omega$}] (v) at (0,0) {};
\draw (v) edge +(-.5,-.5) edge (0,-.5) edge (.5,-.5);
\end{tikzpicture}
\mapsto 
\begin{tikzpicture}
  \node[int] (v) at (0,0) {};
  \node[int,label=90:{$\Delta_1$}] (w) at (0,.7) {};
  \draw (v) edge +(-.5,-.5) edge (0,-.5) edge (.5,-.5) (v) edge (w);
  \end{tikzpicture}
\end{align*}
\end{itemize}

We refer to section \ref{sec:GCs} below for more precise definitions, including signs, prefactors and degrees. We note that these complexes depend on the chosen integer $m$, although this dependence is kept implicit in the notation.
All three of the complexes above are in fact dg Lie algebras, with the Lie brackets defined similarly to $\delta_{glue}$ above, just operating on a pair of decorations on two distinct graphs.

One has natural maps between the above complexes for various $g$ that fit into a commutative diagram of dg Lie algebras
\[
\begin{tikzcd}
  \cdots \ar{r} \ar{d}& \GC_{(g-1),1}^\tp \ar{r}\ar{d} & \GC_{(g),1}^\tp \ar{r}\ar{d}& \GC_{(g+1),1}^\tp \ar{r}\ar{d}& \cdots\ar{d} \\
  \cdots \ar{r} \ar{d}& \GC_{(g-1),1} \ar{r}\ar{d}& \GC_{(g),1} \ar{r}\ar{d}& \GC_{(g+1),1} \ar{r}\ar{d}& \cdots \ar{d}\\
  \cdots & \GC_{(g-1)} &\GC_{(g)} & \GC_{(g+1)} & \cdots
\end{tikzcd}.
\]
Furthermore, one has a natural action of the symplectic or orthogonal group
\[
  \OSp_g=
  \begin{cases}
    \Sp(2g) & \text{for even $m$} \\
    O(g,g) & \text{for odd $m$} 
  \end{cases}
\] 
on all three dg Lie algebras considered.
Moreover, $\GC_{(g)}$ may naturally be extended by a nilpotent, negatively graded Lie algebra $\osp^{nil}_g$ of endomorphisms of $H_g$ that respect the pairing \eqref{equ:Hg pairing}, and we will define below an extended dg Lie algebra 
\[
  \GCex_{(g)} := (\osp_g^{nil}\ltimes \GC_{(g)}, \delta).
\]

All the graph complexes above carry a natural grading  by \emph{weight}, with the weight of a graph with $e$ edges, $v$ vertices and total (homological) decoration degree $Dm$ defined to be the number
\[
W = 2(e-v) + D.
\]
This positive integer valued quantity is preserved by the differentials and Lie brackets. In particular our graph complexes split into a direct product of finite dimensional subcomplexes according to weight.
We shall denote the graded piece of the complexes or cohomology of given weight $W$ by the prefix $\gr^W (\cdots)$.

The purpose of this paper is two-fold. First, we describe the relation between the three dg Lie algebras above, with the following result.

\begin{thm} \label{thm:GC comparison} $ $
  \begin{enumerate}[(i)]
    \item The projection $\GC_{(g),1}^\tp\to \GC_{(g),1}$ defined by setting graphs with tadpoles to zero induces an isomorphism 
    \begin{equation}\label{equ:thm GC comparison 1}
    \gr^W  H(\GC_{(g),1}^\tp) \to \gr^W  H(\GC_{(g),1})
    \end{equation}
    for all $W\geq 2$ and $g\geq 0$. 

    \item For $g\geq 2$ one has a short exact sequence of graded Lie algebras
    \[
    0 \to \pi^\Q S T W_{g} \to H(\GC_{(g),1}) \to H(\GCex_{(g)}) \to 0.
    \]
    with $\pi^\Q S T W_{g}$ being the rational homotopy groups of the unit sphere bundle of the tangent bundle to $W_{g}$.
  \end{enumerate}
\end{thm}
We note that the weight $1$ part of the cohomology of all our graph complexes may be explicitly computed for all genera $g$, see section \ref{sec:weight 1 comparison} below.

As an application we show how the conjecture stated in the introduction of \cite{AKKN} follows from \cite{Matteo}. For this, recall the result by the third author which shows that Drinfeld's Grothendieck-Teichm\"uller Lie algebra $\grt_1$ is isomorphic to the zeroth cohomology of Kontsevich's graph complex $H^0(\GC)$ \cite[Theorem 1.1]{Willgrt}. Moreover, the one-loop approximation of the latter is identified in \cite{severa} with Alekseev and Torossian's Kashiwara-Vergne Lie algebra $\krv$ \cite{AlekseevTorossian}. Furthermore, there is an injective Lie algebra morphism
\begin{equation}\label{thm:grtkrv}
\grt_1 = H^0(\GC) \to \krv.
\end{equation}
The genus one analogues of both $\grt_1$ and $\krv$ have been introduced in the literature. Enriquez defines the elliptic Grothendieck-Teichm\"uller Lie algebra as a semi-direct product
\[
\grt_1^{ell}=\grt_1 \ltimes \mathfrak r_{ell}
\]
where $\mathfrak r_{ell}$ is an explicitly defined Lie subalgebra of the Lie algebra of derivations of the free Lie algebra on two generators \cite{Enriquez}. By \cite[Corollary 3]{Matteo} $\mathfrak r_{ell}$ is moreover isomorphic to $H^0(\GC_{(1)})$. On the other hand, an elliptic variant of $\krv$ was introduced in \cite{AKKN0,AKKN} motivated by string topology of surfaces, and in \cite{SchnepsRaphael} in mould theory. It is also naturally described as a Lie subalgebra of derivations of the free Lie algebra on two generators and is denoted $\krv_{1,1}$. Theorem \ref{thm:grtkrvell} yields the elliptic version of \eqref{thm:grtkrv} conjectured to exist in both \cite{AKKN, SchnepsRaphael}.

\begin{thm}\label{thm:grtkrvell}
The assignment $(\phi,u)\mapsto u$ defines a Lie algebra map $\grt_1^{ell}\rightarrow \krv_{1,1}$.
\end{thm}

The second main topic of this paper is the study of the cohomology of our three dg Lie algebras above for large $g$. To this end we have the following vanishing result.

\begin{thm}\label{thm:main cohom GC vanishing}$ $
\begin{enumerate}[(i)]
  \item For all $g\geq 0$, $W\geq 1$ and $k<(1-m)W$ we have that 
  \[
    \gr^W H^{k}(\GC_{(g),1}^\tp)
    =
    \gr^W H^{k}(\GC_{(g),1})
    =0.
  \] 
  \item (partially contained in \cite[Theorem 82]{Matteo})
  For all $g\geq 2$, $W\geq 1$ and $k<(1-m)W$ 
  \[
    \gr^W H^{k}(\GCex_{(g)}) = 0.
  \]
  \item For all $W\geq 1$, $g\geq W+2$ and $k>(1-m)W$ we have that
  \[
    \gr^W H^{k}(\GC_{(g),1}^\tp)
    =
    \gr^W H^{k}(\GC_{(g),1})
    =
    \gr^W H^{k}(\GCex_{(g)}) = 0.
  \] 
\end{enumerate}
\end{thm}
In other words, for $g\geq W+2$, the cohomology of the weight $W$ part of all three graph complexes becomes concentrated in degree $-(m-1)W$.
From this, one also obtains that all three dg Lie algebras become formal in the limit $g\to\infty$.

A similar result can be obtained for the Chevalley-Eilenberg cohomology $H_{CE}(-)$ of all three dg Lie algebras.
More precisely, we define the Chevalley-Eilenberg complex of any of the three Lie algebras above (say $\fg$ collectively) as the cobar construction of the graded dual 
\[
  C_{CE}(\fg) = \Bar^c \fg^c.
\]
Then $C_{CE}(\fg)$ is a differential graded commutative algebra. It is equipped with an additional grading by weight, inherited from $\fg$.
We denote by $\gr^W H_{CE}(-)$ the weight $W$ piece of the Chevalley-Eilenberg cohomology.

\begin{thm}\label{thm:main CE all}
For any $g\geq 3W$ and $k\neq m W$ we have that 
\[
  \gr^W H_{CE}^k(\GC_{(g),1}) = \gr^W H_{CE}^k(\GC_{(g),1}) = \gr^W H_{CE}^k(\GC_{(g),1}) = 0.
\]
\end{thm}
In other words, for large $g$ the cohomology of the weight $W$ piece becomes concentrated in cohomological degree $mW$.
It follows from Theorem \ref{thm:main CE all} in particular that in the limit $g\to \infty$ the Chevalley-Eilenberg complex becomes formal as a differential graded commutative algebra.
Furthermore, it follows from Theorems \ref{thm:main cohom GC vanishing} and \ref{thm:main CE all} together that as $g\to \infty$ the cohomology and Chevalley-Eilenberg cohomology of our dg Lie algebras form Koszul pairs. 
We refer to sections \ref{sec:CE} and \ref{sec:Koszul} below for more precise formulations of these statements as well as refined degree bounds.

The final question is then to compute the non-vanishing cohomologies for high genera. For this introduction we shall restrict ourselves to the case $m=1$ and the complex $\GCex_{(g)}$ for brevity. Analogous results for the other cases can be found in section \ref{sec:Koszul} below.
We shall give an explicit presentation of the stable cohomology of the Lie algebra. The answer is closely related to Hain's presentation of the Malcev completion of the Torelli group \cite{Hain, Hain2}.

Note that for $m=1$ the cohomology of $\GCex_{(g)}$ becomes concentrated in degree $0$ as $g\to \infty$, and furthermore $\OSp_g=\Sp(2g)$.
Fix a system of fundamental weights $\lambda_1,\dots,\lambda_g$ of $\Sp(2g)$. Denote by $V(\lambda)$ the irreducibel $\Sp(2g)$-representation of heighest weight $\lambda$. In particular $V(\lambda_1)\cong H^1(W_g)\cong H^1(W_{g,1})$. 
We note that one has the following decomposition into irreducible $\Sp(2g)$-representations, for $g\geq 3$. 
\[
\wedge^3(V(\lambda_1))
\cong
V(\lambda_1) \oplus V(\lambda_3)
\]
The projection to $V(\lambda_1)$ is obtained by the contraction of two of the three factors $V(\lambda_1)$ with the given bilinear form.
We may easily identify 
\[
\gr^1 H^0(\GCex_{(g)}) \cong V(\lambda_3),
\]
by identifying $\wedge^3 V(\lambda_1)$ with the space of graphs with one vertex and three decorations.
\[
  \begin{tikzpicture}
    \node[int,label=90:{$\alpha\beta \gamma$}] {};
  \end{tikzpicture}
  \quad \text{with $\alpha,\beta,\gamma\in H_m(W_{g})$}
\]
We may furthermore decompose, for $g\geq 6$,
\[
\wedge^2 V(\lambda_3)
\cong 
V(0) \oplus V(\lambda_2) \oplus V(\lambda_4) \oplus V(\lambda_6) \oplus V(2\lambda_2)\oplus V(\lambda_4+\lambda_2).
\]

Let $R_{(g)}$ be the $\Sp(2g)$-invariant complement of $V(2\lambda_2)$ in $\wedge^2 V(\lambda_3)$.
Denote by 
\[
  \ft_{(g)} = \FreeLie(V(\lambda_3) ) / R_{(g)}
\]
the Lie algebra generated by $V(\lambda_3)$ with with the quadratic relations $R_{(g)}\subset \wedge^2 V(\lambda_3)$.
We equip $\ft_{(g)}$ with a grading by assigning the generators degree $1$.

\begin{thm}\label{thm:main t}
  The inclusion $V(\lambda_3)\cong \gr^1 H^0(\GCex_{(g)})\subset H^0(\GCex_{(g)})$ extends to a Lie algebra homomorphism
  \[
    \ft_{(g)} \to H^0(\GCex_{(g)}).
  \]
  This Lie-algebra homomorphism induces an isomorphism on the graded components 
  \[
    \gr^W\ft_{(g)} \xrightarrow{\cong}
    \gr^W H^0(\GCex_{(g)}).
  \]
  as soon as $g\geq 3W\geq 6$.
\end{thm}

To study the Chevalley-Eilenberg cohomology let us next define a graded commutative algebra $A_{(g)}$ as the Koszul dual commutative algebra of the Lie algebra $\ft_{(g)}$. Concretely, one has a presentation 
\[
  A_{(g)} = S( V(\lambda_3)[-1]) / (R_{(g)})^\perp,
\]
where $V(\lambda_3)\cong V(\lambda_3)^*$ is identified with the dual space of the space of generators of $\ft_{(g)}$ and $(R_{(g)})^\perp$ is the annihilator of $R_{(g)}$ in $S^2( V(\lambda_3)[-1])\cong \wedge^2 V(\lambda_3)^*[-2]$. Note also that $A_{(g)}$ inherits an additional grading by weight.

\begin{thm}\label{thm:main HCE}
One has a zigzag of morphism of dg commutative algebras 
\[
  C_{CE}(\GCex_{(g)}) \leftarrow \bullet \to A_{(g)}
\]
that induces an isomorphism in cohomology
\[
 \gr^W H_{CE}(\GCex_{(g)})
 \cong
 \gr^W (A_{(g)})
\]
as soon as $g\geq 3W\geq 6$.
\end{thm}

The Lie and commutative algebras above appear in various contexts and places in the literature. For example, in \cite{Matteo} it is shown that in the case $m=1$ one has that $\sp_g \ltimes H^0(\GCex_{(g)})$ is identified with a genus-$g$-version of the Grothendieck-Teichmüller Lie algebra. 
The Chevalley-Eilenberg cohomology of $\GC_{(g)}^\tp$ appears as the cohomology of the space $X_1(g)$ of \cite[section 4.3]{KRW} (for large $g$), that captures the major part of the cohomology of the Torelli subgroup of the boundary preserving framed diffeomorphism group of $W_{g,1}$.
Finally the Lie algebra $\ft_{(g)}$ (in our notation) can be identified with a relative Malcev completion $\mathfrak u_g$ (in Hain's notation) of the mapping class group (for $g\geq 6$) as computed by Hain in \cite[Theorem 7.7]{Hain2}.
One may then reformulate a less precise version our results above as follows.
\begin{cor}\label{cor:Hain}
The Malcev completion $\mathfrak u_g\cong \ft_{(g)}$ of the genus $g$ mapping class group relative to the symplectic group as computed by R. Hain in \cite[Theorem 7.7]{Hain2} is Koszul in the limit $g\to \infty$.
\end{cor}

This means more precisely that the Koszulness condition for $\ft_{(g)}$ is satisfied in a range of weights that tends to $\infty$ as $g\to\infty$.


\subsection*{Ackowledgements}
The third author thanks Alexander Kupers and Oscar Randal-Williams for helpful discussions.

We shall also note that results partially overlapping ours have independently been obtained by Kupers and Randal-Williams \cite{KRWnew}.

\section{Notation and preliminaries}
\subsection{Notation}\label{sec:notation}
Unless otherwise stated all vector spaces are taken over the rationals $\Q$. 
We abbreviate the term differential graded by dg. We always use cohomological conventions, that is, differentials have degree $+1$, and we use $\mathbb Z$-gradings unless otherwise stated.
For $V$ a graded vector space we denote by $V[k]$ the same graded vector space with degrees shifted downwards by $k$ units. For example, if $V$ is concentrated in degree $0$, then $V[k]$ is concentrated in degree $-k$.

Almost all objects we consider will be graded objects in dg vector spaces or similar categories. That is, these objects come with two gradings, the cohomological grading and an additional (``weight") grading. 
Concretely, we will consider two incarnations of the additional grading, namely a graded dg vector space $V$ may be written either as a direct sum
\[
  V = \bigoplus_k \gr^k V 
\]
or as a direct product
\[
  V = \prod_k \gr^k V 
\]
of dg sub-vector spaces $\gr^k V \subset V$.
We will call the second type of grading \emph{complete gradings}.
For example, the dual vector space $V^*$ of a dg vector space with additional grading $V=\bigoplus_k \gr^k V$ has a complete grading 
\[
  V^* = \prod_k (\gr^k V)^*.
\]
Often it is hence helpful to consider instead the graded dual dg vector space 
\[
V^c := \bigoplus_k (\gr^k V)^*. 
\]

If $V$ is equipped with further algebraic structure, for example a dg Lie or dg commutative algebra structure, then we will say that the additional grading is compatible with that algebraic structure if the defining algebraic operations (say the Lie bracket or commutative product) restrict to morphisms of dg vector spaces 
\[
\gr^k V \times \gr^l V \to \gr^{k+l} V.
\]

We shall consider several forms of truncations of dg objects with additional grading.
For $\alpha$ an integer, we denote by
\begin{equation}\label{equ:H sq alpha}
  H^{[\alpha]}V = \bigoplus_k \gr^k H^{k\alpha}(V)[-k\alpha]
\end{equation}
the bigraded vector space whose part of additional degree $k$ is concentrated in cohomological degree $\alpha k$ and agrees with the cohomology of $\gr^k V$ there.
We define for a dg vector space the truncation  
\[
 (\tru^{\leq k}V)^j
 =
\begin{cases}
  V^j & \text{for $j<k$} \\
  \{x\in V^k\mid dx=0\} & \text{for $j=k$} \\
  0 & \text{for $j>k$}
\end{cases}.
\]
Here $V^j$ refers to the part of $V$ of cohomological degree $j$.
Furthermore, for a dg vector space $V$ with additional grading and $\alpha$ an integer we set 
\begin{equation}\label{equ:tru sq alpha}
  \tru^{[\alpha]}V = \bigoplus_k \gr^k \tru^{\leq k\alpha}(V).
\end{equation}
We note that if $V$ carries further algebraic structure (say a dg Lie or dg commutative algebra structure), compatible with the additional grading, then $\tru^{[\alpha]}V$ and $H^{[\alpha]}(V)$ inherit that structure.
Furthermore, one always has a zigzag of dg vector spaces with additional grading 
\[
V \leftarrow  \tru^{[\alpha]}V \to H^{[\alpha]}(V)
\]
given by the natural inclusion and projection.
This zigzag also preserves the given algebraic structure on $V$ if present.

\subsection{Representation theory of the symmetric and orthogonal groups}\label{sec:rep theory}
We use the notation 
\[
  \OSp_g=
  \begin{cases}
    \Sp(2g) & \text{for even $m$} \\
    O(g,g) & \text{for odd $m$} 
  \end{cases}
\] 
to either denote the symplectic group, or the orthogonal group associated to the non-degenerate bilinear form of signature $(g,g)$.
We understand $\OSp_g$ as an algebraic group over $\Q$, and we denote by $\OSp_g(\K)$ the $\K$-points for a field $\K\supset \Q$.
We shall recall some standard facts on the representation theory of these groups. We refer to \cite[section 2]{KRWTor} for a beautiful overview that is perfectly suited for the present paper.

Every finite dimensional representation of $\OSp_g$ decomposes into irreducible representations (cf. \cite[Proposition 22.41]{Milne}).
Let $V_g:= \Q^{2g}$ be equipped with the standard symplectic form for $m$ odd or the standard non-degenerate symmetric form of signature $g,g$ for $m$ even. Slightly abusively, we also denote by $V_g$ the defining representation of $\OSp_g$.
Every finite dimensional irreducible representation of $\OSp_g$ is contained in a tensor product $V_g^{\otimes n}$ for some $n$ (see \cite[Theorem 4.14]{Milne}).
We will say that a finite dimensional representation $V$ of $\OSp_g$ is \emph{of order $n$} if it is a subquotient of the representation $V_g^{\otimes n}\otimes U$, for $U$ some finite dimensional vector space, considered as a trivial representation. Equivalently, this means that in the decomposition of $V$ into irreducibles only those irreducibles appearing in $V_g^{\otimes n}$ are present.
We say an algebraic representation $V$ of $\OSp_g$ is \emph{of order $\leq n$} if it is a direct sum of representations of order $0,1,2,\dots, n$.


\begin{lemma}\label{lem:van condition}
Let $V$ be a representation of $\OSp_g$ of order $n$.
Then $V=0$ iff the space of invariants 
\[
   (V\otimes V_g^{\otimes n})^{\OSp_g}=0
\]
vanishes.
\end{lemma}
\begin{proof}
  We need to check that for $V$ non-zero the invariant space above is non-zero. 
  We may pick a surjective map of representations $\pi: V_g^{\otimes n}\otimes U\to V$. This is a non-zero element in 
  \begin{align*}
   ( (V_g^{\otimes n}\otimes U)^* \otimes V )^{\OSp_g}
   \cong  
   ( V_g^{\otimes n} \otimes V )^{\OSp_g} \otimes U^*,
  \end{align*}
using that $V_g$ is self-dual due to the presence of the non-degenerate pairing. But the presence of a the non-zero element $\pi$ implies that the first factor in the tensor product on the right-hand side is not zero.
\end{proof}
Note that we allow the zero vector space to be a (necessarily trivial) representation by convention.
We will use the above Lemma in the following form.
\begin{lemma}\label{lem:complex inv}
Let
\[
  C = \cdots \xrightarrow{d} C^{k-1} \xrightarrow{d}  C^k \xrightarrow{d} C^{k+1} \xrightarrow{d} \cdots
\]
be a dg vector space with a compatible action of $\OSp_g$. Suppose that the degree $k$-subspace $C^k$ is a finite dimensional representation of order $\leq n$ (in the sense above).
Then $H^k(C)=0$ iff 
\[
H^k \left( (C \otimes V_g^{\otimes M})^{\OSp_g}  \right) =0
\]
for any $M\leq n$.
\end{lemma}
\begin{proof}
By the assumption $H^k(C)$ is of order $\leq n$, and hence by the previous Lemma we need to check that 
\[
  0=(H^k(C)\otimes V_g^{\otimes M})^{\OSp_g}
  =H^k(C \otimes V_g^{\otimes M})^{\OSp_g}.
\]
Using Schur's Lemma one can see that taking invariants of a complex of completely reducible representations commutes with taking cohomology. 
Hence 
\[
  H^k \left( C \otimes V_g^{\otimes M}  \right)^{\OSp_g}
  = 
  H^k \left( (C \otimes V_g^{\otimes M})^{\OSp_g}  \right)
\]
and we are done.
\end{proof}

Let us also remark that $\OSp_g(\Q)\subset \OSp_g(\C)$ is Zariski dense (this can be seen from \cite[Corollary 18.3]{BorelBook}).
Hence the complexification of the space of invariants of a (rational) representation $V$ can be identified with the invariants of the complexification, so that one has isomorphisms 
\[
  V^{\OSp_g}\otimes \C= V^{\OSp_g(\Q)}\otimes \C\cong (V\otimes \C)^{\OSp_g(\Q)}
  \xhookleftarrow{\cong} (V\otimes \C)^{\OSp_g(\C)}.
\]
One may use this observation to almost always work with complex representations of the complex semisimple Lie group $\OSp_g(\C)$ if desired, forgoing the need to use results from algebraic group theory.

Second, we shall need the classical invariant theory for the groups $\OSp_g$.
We directly state a version with the appropriate degree shifts appearing in our application.
To this end let $\Delta_1\in (V_g[-m]\otimes V_g[-m])^{S_2}$ be the symmetric element that is dual to the given canonical bilinear form on $V_g$. Next consider the action of the symmetric group $S_{2N}$ on the tensor product $V_g[-m]^{\otimes 2N}$ by permuting factors.
The action on the element $\Delta_1^{\otimes N}$ then gives a map 
\begin{gather*}
S_{2N} \to V_g[-m]^{\otimes 2N}\\
\sigma \mapsto \sigma \cdot \Delta_1^{\otimes N}.
\end{gather*}
This map factorizes over the cosets $S_{2N}/(S_2\wr S_N)$ (with $\wr$ denoting the wreath product).
Furthermore, the image is obviously $\OSp_g$-invariant, since so is $\Delta_1$.

\begin{thm}[A version of the First and Second Fundamental Theorems of Classical Invariant Theory]\label{thm:inv theory}
The map
\[
 \Q[ S_{2N}/(S_2\wr S_N) ] [-2Nm] \to \left( V_g[-m]^{\otimes 2N} \right)^{\OSp_g}
\] 
is surjective for all $g$, $N$, and an isomorphism for $g\geq N$.
\end{thm}
The statement can be found in the present form as \cite[Theorem 2.6]{KRWTor}. The statement for $\Sp(2g,\Q)$ is also \cite[Theorems 9.5.9, 9.5.11]{Loday}, going back to \cite{Procesi}. One also finds the analogous statement for $O(2g)$ in \cite[Theorems 9.5.2, 9.5.5]{Loday}, from which the statement for $O(g,g)$ may be obtained via complexification.

\subsection{Bar and cobar construction and Koszul duality for Lie and commutative algebras}
The bar and cobar construction form a pair of adjoint functors between the category of augmented algebras over a Koszul operad $\POp$, and the category of conilpotent coaugmented dg coalgebras over the Koszul dual cooperad, see \cite[chapter 11]{lodayval}.
We shall only need the case where $\POp$ is either the operad $\Lie$ governing Lie algebras or the operad $\Com$ governing (non-unital, or equivalently augmented unital) commutative algebras, with the Koszul dual cooperads those governing cocommutative or Lie coalgebras.
In this case we have adjoint pairs 
\begin{align*}
  \Bar_\Lie : \dg \Lie \Alg &\rightleftarrows \dg \Com\Alg^c : \Bar^c_{\Com} \\
\Bar_\Com : \dg \Com \Alg &\rightleftarrows \dg \Lie\Alg^c : \Bar^c_{\Lie} 
\end{align*}
between the category of dg Lie algebras (resp. augmented dg commutative algebras) on the left-hand side and the category of conilpotent coaugmented dg cocommutative coalgebras (resp. conilpotent dg Lie coalgebras) on the right-hand side.
Concretely, one has that for a dg Lie algebra $\fg$
\[
  \Bar_\Lie \fg = (\Free^c_\Com( \fg[1] ), d) \cong (S(\fg[1]),d)
\]
is the free cocommutative coalgebra cogenerated by $\fg[1]$, with the Chevalley-Eilenberg differential.
Similarly, for a coaugmented dg cocommutative coalgebra $C$ one has that 
\[
  \Bar^c_{\Com} C = (\Free_{\Lie}(\bar C[-1]), d)
\]
is the free Lie algebra generated by the degree shifted coaugmentation coideal $\bar C$, equipped with the Harrison differential.
The other two cases are obtained by swapping the role of $\Com$ and $\Lie$.
If no confusion can arise we will often omit the subscript of the bar/cobar functors and just write $\Bar A$ which shall mean $\Bar_{\Lie} A$ if $A$ is a dg Lie algebra and $\Bar_\Com A$ if $A$ is an augmented dg commutative algebra.

It is known that the adjunction units and counits $\Bar^c\Bar A\to A$ and $C\to \Bar \Bar^c C$ are quasi-isomorphisms, see \cite[Proposition 2.5 and Theorem 2.6]{valHTHA}.
Furthermore, the functor $\Bar$ sends quasi-isomorphisms to quasi-isomorphisms \cite[Proposition 11.2.7]{lodayval}. The functor $\Bar^c$ does not preserve quasi-isomorphisms in general. However, it does preserve quasi-isomorphisms in most "good cases". For example if, as in all relevant cases for us, the quasi-isomorphism preserves an additional grading, and domain and codomain are degree-wise finite-dimensional, then $\Bar^c X\cong (\Bar X^c)^c$, so that one can use that $\Bar$ preserves quasi-isomorphisms.

Next we consider a (non-differential) graded Lie algebra $\fg$ defined by a homogeneous quadratic presentation
\[
\fg = \Free_\Lie(V) / \langle R\rangle 
\]
with the generating graded vector space $V$, the space of relations 
\[
R\subset \wedge^2 \fg  
\]
and $\langle R\rangle$ the ideal generated by $R$.
Such a graded Lie algebra automatically has an additional grading by the number of generators, i.e., we declare any $k$-ary bracket of generators to live in additional degree $k$.
This grading is then inherited by the bar construction $\Bar \fg$.

For simplicity, and since this is true in all our cases, we assume that the generators $V$ are concentrated in the single cohomological degree $\alpha$. Then $\gr^W \fg$ is concentrated in degree $\alpha W$.
Furthermore $\gr^W\Bar\fg$ is concentrated in cohomological degrees 
\[
  (\alpha-1)W,\dots , \alpha W-1.
\]
In particular, one has a map of dg cocommutative coalgebras 
\[
H^{[\alpha-1]}(\Bar \fg) \to \Bar\fg,
\]
using the notation of section \ref{sec:notation}.
The graded cocommutative coalgebra $C=H^{[\alpha-1]}(\Bar \fg)$ is called the \emph{Koszul dual} of $\fg$. Looking at the definition of the bar construction, $C$ also has a homogeneous quadratic presentation, with the space of cogenerators $V[1]$, and the space of quadratic corelations given by the annihilator $R^\perp$ of $R$.
We also consider the graded dual dg commutative algebra $C^c$, and call it the Koszul dual dg commutative algebra of $\fg$.
It then has the quadratic presentation 
\[
C^c = \Free_{\Com}(V^*[-1])/\langle R^\perp \rangle  
\]

The same reasoning applies also if we invert the roles of the commutative and Lie operads.

We will usually consider only finite dimensional spaces of generators $V$. In this case the relation of being the Koszul dual is clearly reflexive, that is, the Koszul dual of $C^c$ above is again $\fg$.

\begin{defi}\label{def:Koszul}
Let $\fg=\Free_\Lie (V)/\langle R \rangle$ be a graded Lie algebra with space of generators $V$ concentrated in cohomological degree $\alpha$ and the homogeneous quadratic space of relations $R$.
Then we say that $\fg$ is \emph{Koszul} if the canonical map from the Koszul dual cocommutative coalgebra
\[
  C=H^{[\alpha-1]}(\Bar \fg) \to \Bar\fg
\]
is a quasi-isomorphism.

Likewise, a commutative algebra $A=\Free_\Com (V)/\langle R \rangle$ given by a homogeneous quadratic presentation with $V$ concentrated in cohomological degree $\alpha$ is \emph{Koszul} if the canonical morphism of dg Lie coalgebras
\[
  \fc := H^{[\alpha-1]}(\Bar A) \to \Bar A
\]
from the Koszul dual graded Lie coalgebra $\fc$ is a quasi-isomorphism.
\end{defi}

Note that, trivially, if $\fg$ is Koszul then $\Bar \fg$ is formal as a dg cocommutative coalgebra.

\begin{lemma}
Let $\fg=\Free_\Lie (V)/\langle R \rangle$ be a quadratically presented Lie algebra as above, with $V$ finite dimensional, and $A=\Free_{\Com}(V^*[-1])/\langle R^\perp \rangle$ the Koszul dual graded commutative algebra. Then the following are equivalent:
\begin{itemize}
\item $\fg$ is Koszul.
\item $A$ is Koszul.
\end{itemize} 
\end{lemma}
\begin{proof}
We only show one direction, the other follows by symmetry.
Koszulness of $\fg$ means that 
\[
  H^{[\alpha-1]}(\Bar \fg) \to \Bar\fg 
\]
is a quasi-isomorphism. Taking graded duals $(-)^c$ on both sides and using the finite dimensionality of $V$ this is equivalent to
\[
  \Bar^c \fg^c \to A
\]
being a quasi-isomorphism. Taking the bar construction, and using that the bar construction preserves quasi-isomorphisms, we find that 
\[
  \Bar \Bar^c \fg^c \to \Bar A
\]
is a quasi-isomorphism. But hence so is the composition 
\[
  \fg^c \to \Bar \Bar^c \fg^c \to \Bar A,
\]
so that $A$ is Koszul.
\end{proof}

In the case that $A$ is the Koszul dual of $\fg$ and either (and hence both of) $\fg$ and $A$ are Koszul, we will call $(\fg,A)$ a Koszul pair.

We will encounter Koszul objects in the following form, except for the additional complication that quasi-isomorphisms only hold in a range of additional degrees. (See also Proposition \ref{prop:cohom conc follows koszul 2} below.)

\begin{prop}\label{prop:cohom conc follows koszul}
Suppose that $\fg$ is a dg Lie algebra with an additional positive grading with the property that there is an $\alpha$ such that $\gr^W H(\fg)$ is concentrated in cohomological degree $\alpha$ for any $W$. Then $\fg$ is formal.

Suppose in addition that the cohomology $\gr^W H(\Bar \fg)$ of the graded pieces of the bar construction is concentrated in degree $(\alpha-1)W$ for any $W$, and that the $\gr^W \fg$  are finite dimensional.
Then $H(\fg)$ is Koszul with Koszul dual graded commutative algebra $A:=H(\Bar \fg)^c$. Furthermore, $A$ and $H(\fg)$ have quadratic presentations of the form 
\begin{align*}
H(\fg) &\cong \Free_\Lie(\gr^1 H(\fg)) / \langle \gr^2 A^c \rangle
\\
A &\cong \Free_\Com(\gr^1 A) / \langle \gr^2 H(\fg)^c \rangle\, .
\end{align*}
\end{prop}
\begin{proof}
For the first statement note that under the assumptions the zigzag 
\[
  \fg \leftarrow \tru^{[\alpha]}\fg \to H^{[\alpha]}(\fg) = H(\fg)
\]
realizes the formality.

For the second statement, we can show in exactly the same way that $(\Bar\fg)^c$ is formal as a dg commutative algebra.
Hence we find that 
\[
H(\fg) \xleftarrow{\sim} \Bar^c\Bar H(\fg) \to \bullet \leftarrow  \Bar^c A^c
\] 
is connected by a zig-zag of quasi-isomorphisms to the cobar construction of $\Bar^c A^c$, and all arrows preserve the additional grading. By assumption $\gr^W A^c$ is concentrated in degree $(\alpha-1)W$, and by the assumption that the additional grading is positive (i.e., $W\geq 1$) one has that $\gr^W\Bar^c A^c$ is concentrated in cohomological degrees $(\alpha-1)W,\dots,\alpha W$. But $\gr^W H(\fg)$ is concentrated in degree $\alpha W$ and hence must be identified with the top piece of the cohomology of $\gr^W\Bar^c A^c$. But this top piece is (by the definition of the cobar construction) the graded Lie algebra generated by $\gr^1 \bar A^c[-1]\cong \gr^1 H(\fg)$, with relations 
being the image of $\gr^2 \bar A^c$ under the coproduct 
\begin{equation}\label{equ:rel inj}
\gr^2 \bar A^c[-2] \to S^2(\gr^1\bar A^c)[-2] \cong \wedge^2 (A^c[-1]).    
\end{equation}
We may proceed in the same manner for $A$ and see that $A$ is connected to $\Bar^cH(\fg)^c$ by a zigzag of quasi-isomorphisms preserving the additional gradings 
\[
  A \xleftarrow{\sim} \Bar^c\Bar A \to \bullet \leftarrow  \Bar^c H(\fg)^c,
\]
and that $A$ also has a quadratic presentation as indicated. This also shows that $A$ is generated by $\gr^1 A$, so that in particular the map \eqref{equ:rel inj} is injective. (...and likewise the analogous map realizing $\gr^2 H(\fg)^c$ as the space of generators for the presentation of $A$.) 

Finally, Koszulness of $H(\fg)$ means that $\gr^W H(\Bar H(\fg))$ is concentrated in degree $(\alpha-1)W$, and that follows from the fact that $\Bar H(\fg)$ is quasi-isomorphic to $A^c$ as seen above, and the assumption on $A$.
\end{proof}

\begin{rem}
The material in this section is standard, but hard to reference in the present form. More commonly, one considers Koszul duality for associative algebras instead of commutative and Lie algebras.
Concretely, one may just replace $\Com$ and $\Lie$ above by the associative operad $\Ass$ to recover the standard notion of Koszulness.
Furthermore, if a commutative algebra $A$ with a homogenous quadratic presentation (and say degree-wise finite dimensional) is Koszul in our sense then it is Koszul as an associative algebra.
To see this, note that the associative cobar construction and the commutative cobar construction are related by the universal enveloping algebra construction $U$, i.e., $\Bar_{\Ass}^c A^c = U(\Bar_{\Com}^c A^c)$. Furthermore, $U(-)$ preserves quasi-isomorphisms and the grading. (Since on the level of complexes $U$ is just the symmetric product by the Poincar\'e-Birkhoff-Witt Theorem.)
Hence $\gr^W H(\Bar_{\Ass}^c A^c)$ is concentrated in degree $(1-\alpha)W$ iff so is $\gr^W H(\Bar_{\Com}^c A^c)$.
\end{rem}

\section{Graph complexes}\label{sec:GCs}

\subsection{Definitions of various graph complexes}
In this section we give a combinatorial definition of the graph complexes considered in this paper. 
These complexes have appeared at other places in the literature, for example \cite{CamposWillwacher, Matteo}.

We say that a (directed) graph with $n$ vertices and $k$ edges is an ordered set of $k$ pairs $(i,j)$ of numbers $i,j\in \{1,\dots,n\}$. We say that these $k$ elements are the edges of the graph, with the edge $(i,j)$ pointing from vertex $i$ to vertex $j$. At this point we allow arbitrary sets of edges, in particular tadpoles (or short cycles), that is edges of the form $(i,i)$, and we do not ask that the graph is connected.
We denote the set of such graphs with $n$ vertices and $k$ edges by $\gra_{n,k}$. This set carries an action of the group $S_n\times (S_2 \wr S_k)$, with $S_n$ acting by renumbering vertices, $S_k$ by reordering edges and the $S_2$ by changing the directions of the edges, i.e., by flipping the two members of the pairs. 

Let furthermore $V$ be any finite dimensional graded vector space and $m$ an integer.
We then define a graded vector space of coinvariants
\begin{equation}\label{equ:fGdef}
  \fG_{V,2m} := \bigoplus_{n\geq 1, k\geq 0}
  \left( \Q\gra_{n,k}[2mn]\otimes \Q[1-2m]^{\otimes k} 
  \otimes (S(V))^{\otimes n} \right)_{S_n\times (S_2 \wr S_k)}.
\end{equation}
Here the group $S_n$ acts diagonally on the vector space $\Q\gra_{n,k}$ generated by the set $\gra_{n,k}$ and the $n$ symmetric product factors $S(V)$, and $S_k$ acts diagonally on $\Q\gra_{n,k}$ and by permuting the factors $\Q[1-2m]$, with appropriate Koszul signs.
We may interpret elements of $\fG_{V,2m}$ as linear combinations of isomorphism classes of graphs, whose vertices may be decorated by zero or more elements of $V$. Additionally, such graphs come with an ordering of the edges and the decorations, and we identify two such ordering up to sign.

We say that the valence of a vertex in a graph is the number of incident half-edges, plus the number of decorations.
(E.g., a vertex $v$ with $3$ incident half-edges and decoration in $S^2(V)$ has valence $5$.)

We define on $\fG_{V,2m}$ a differential $d_c'$ defined such that for a graph $\Gamma\in \fG_{V,2m}$
\[
  d_c' \Gamma = \sum_{e}  \pm \Gamma / e,
\]
with the sum running over non-tadpole edges $e$ and the graph $\Gamma/e$ obtained by contracting the edge $e$.
We should remark on several points:
\begin{itemize}
  \item Mind that elements of $\fG_{V,2m}$ are equivalence classes of linear combinations of graphs with additional tensor factors.
  When we say ``...for a graph $\Gamma\in \fG_{V,2m}$" we actually mean that we take a representative of one equivalence class, and we take a linear combination with one term. 
  \item The operation of contracting some edge $e=(i,j)$ is precisely the following. We may suppose, using the action of the permutation group, that $(i,j)=(1,2)$, and that the edge $e$ is the first in the ordering. 
  Then we remove the edge $e$ from the list of edges.
  We then remove vertex $2$ and replace all occurrences of "$2$" by $1$ in the list of edges. Then we renumber the vertices again by $1,2,\dots,n-1$. Finally we alter the decorations by applying the natural map 
  \[
  S(V)^{\otimes n}\to  S(V)^{\otimes n-1} 
  \] 
  by multiplying the first two factors, using the commutative product on the symmetric algebra $S(V)$.
  With this ordering, the sign of the respective term in the formula above is ``$+$".
\end{itemize}

It is an exercise to check that $(d_c')^2=0$.

We also note that $(\fG_{V,2m},d_c')$ is naturally a graded commutative algebra, with the product given by disjoint union of graphs.

Next we assume that we have an element (the ``diagonal") of degree $2m$ 
\[
\Delta  \in S^2(V\oplus \Q 1 ),
\]
with $1$ a symbol that we shall interpret as the unit in $S(V)$. 
We then define a second differential $d_{cut}$ such that 
\[
d_{cut} \Gamma = \sum_e \pm \cut(\Gamma,e)   
\]
where the sum is again over edges and $\cut(\Gamma,e)$ is defined as follows:
We assume that $e=(1,2)$ and that it is the first edge in the ordering. Then $\cut(\Gamma,e)$ is obtained by removing $e$, and multiplying $\Delta$ from the left into $S(V)^{\otimes n}$, at positions $1$ and $2$ of the tensor product. The sign is then ``+", but might that there may arise Koszul signs in the multiplication. 

Again it is an exercise to check that $d_{cut}^2=(\pm d_c'\pm d_{cut})^2=0$.


We now specialize further to the situation at hand and let $V=\bar H^\bullet(W_{g})$, and $\Delta \in S^2(H^\bullet(W_{g}))$ be the canonical diagonal element.
To be concrete we may fix a basis $\{1,a_1,\dots, a_g,b_1,\dots,b_g, \omega\}$ of $H^\bullet(W_{g})$ such that the Poincar\'e duality pairing is 
\begin{align*}
\langle 1,\omega\rangle &= 1
&
\langle a_i,b_j\rangle &= \delta_{ij},
\end{align*}
with all other pairings zero.
Then 
\begin{equation}\label{equ:Delta explicit}
\Delta = 1\otimes \omega +\omega\otimes 1
+ 
\underbrace{
  (-1)^m \sum_{i=1}^g (a_i\otimes b_i + (-1)^m b_i\otimes a_i)}_{=:\Delta_1}.
\end{equation}

Then we define the operation $d_{mul}$ (the third and final part of the differential), that acts via the formula
\[
d_{mul} \Gamma = \sum_v \pm mul(v,\Gamma),  
\]
with the sum being over vertices with a single attached edge, and $mul(v,\Gamma)$ is obtained by the following procedure:
\begin{itemize}
  \item Replace the decorations of $v$ in $H(W_g)$ by their product in $H(W_g)$.
  \item Contract the edge adjacent to $v$ as in the definition of $d_c'$.   
\end{itemize}

Note that by degree reasons we obtain a nonzero contribution only if $v$ has at most two decorations.
We shall furthermore split 
\[
d_{mul} = d_{mul}' + d_{mul}''
\]
with $d_{mul}'$ containing the contributions for which $v$ has $0$ or $1$ decorations, and $d_{mul}''$ those for which $v$ has two decorations. Note that the term $d_{mul}'$ is identical to some summands of the contraction differential $d_c'$, while $d_{mul}''$ creates at least one $\omega$-decoration.

We consider the operation 
\[
d = -d_c' + d_{cut} + d_{mul}
\]
on $\fG_{V,2m}$. One may check that $d^2\neq 0$, and that more precisely $d^2\Gamma$ is obtained by summing over all tadpoles with a single adjacent edge, and replacing the tadpole by an $\omega$-decoration, multiplied by $2+(-1)^m 2g$. Pictorially:
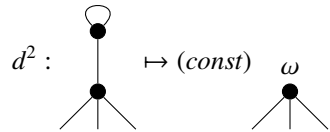
\[
d^2:
\begin{tikzpicture}
  \node[int] (v) at (0,.4) {};
  \node[int] (w) at (0,-.4) {};
\draw (v) edge (w) (v) edge[loop] (v)
(w) edge +(-.5,-.5) edge +(.5,-.5) edge +(0,-.5);
\end{tikzpicture}
\mapsto
(\mathit{const})
\begin{tikzpicture}
  \node[int,label=90:{$\omega$}] (w) at (0,-.4) {};
\draw 
(w) edge +(-.5,-.5) edge +(.5,-.5) edge +(0,-.5);
\end{tikzpicture}
\]
We may however obtain a complex on which $d^2=0$ in one of three ways.
\begin{itemize}
\item We define the quotient complex 
\[
  \fG_{(g),1}^\tp = \fG_{V,2m} / I 
\]
by dividing by the ideal formed by graphs that contain an $\omega$-decoration, or a vertex of valence $\leq 2$.
On this quotient complex $d^2=0$ since $d^2$ creates one $\omega$-decoration.
\item We define the subcomplex 
\[
  \fG_{(g),1}\subset \fG_{(g),1}^\tp
\]
that furthermore contains no tadpoles.
\item We define the subquotient of $\fG_{V,2m}$
\[
  \fG_{(g)}
\] 
by taking the subcomplex containing no tadpoles, and quotienting by the graphs that contain at least one vertex of valence $\leq 2$.
\end{itemize}

\begin{lemma}
The above subquotients of $(\fG_{V,2m},d)$ are well-defined and in particular the differential squares to zero on these subquotients.
\end{lemma}
\begin{proof}
  We first claim that $I$ as in the definition is a subcomplex. (It is clear that it is an ideal.)
First we note that the differential $d$ cannot remove a decoration by $\omega$ on a graph in $\fG_{V,2m}$. 
We then have to check that the differential $d$ cannot remove vertices of valence $\leq 2$. 
Since $d$ acts "locally" on only two vertices, it is clear that for a graph $\Gamma$ with three or more vertices of valence $\leq 2$, $d\Gamma$ contains no graphs that have no vertices of valence $\leq 2$. Furthermore, if $\Gamma$ contains exactly two vertices of valence $\leq 2$ then in the "worst case" the differential can contract an edge between them, and hence still produce a vertex of valence $\leq 2$.
So we are left with the case that $\Gamma$ contains a unique vertex $v$ of valence $\leq 2$, and we need to consider the parts of $d$ affecting $v$. The piece $d_{cut}$ cannot increase the valence of vertices anyway. But the terms of the piece $d_c'$ that would remove a single vertex $v$ of valence $\leq 2$ are cancelled by the contribution $d_{mul}'$. Hence we conclude that $dI\subset I$ as desired.
Furthermore, since $d^2$ creates an $\omega$-decoration, we have that $d^2\Gamma\subset I$ for all $\Gamma\in \fG_{V,2m}$, and hence $\fG_{(g),1}^\tp = \fG_{V,2m} / I$ is a well defined dg vector space.

Next, the differential can clearly not create tadpoles: The only piece that potentially could is the contraction part $d_c'$ acting on an edge of a double edge. But graphs with multiple edges have odd symmetries by permuting these edges, and hence are zero. Hence no tadpoles can be created by the differential and $\fG_{(g),1}\subset \fG_{(g),1}^\tp$ is indeed a subcomplex.

For $\fG_{(g)}$ one notes that if a graph $\Gamma$ has no tadpole, then $d^2\Gamma=0$. So the subspace of $\fG_{V,2m}$ spanned by graphs without tadpoles is a complex.
One then shows in the same way as above that the quotient by the subspace spanned by graphs with vertices of valence $\leq 2$ is well-defined, to see that $\fG_{(g)}$ is a well-defined dg vector space.
\end{proof}

For later use we shall also define 
\[
d_c = d_{mul}'-d_c',  
\]
thus removing the pieces from $d_c'$ that are cancelled by $d_{mul}'$. 

Suppose that $\Gamma$ is a graph in one of the three graph complexes $\fG_{(g),1}$, $\fG_{(g),1}^\tp$ or $\fG_{(g)}$ above, with $e$ edges, $v$ vertices and a total decoration degree $mD$. (In other words, $\Gamma$ has $D$ decorations if we count the top class $\omega$ as two decorations.) Then the cohomological degree of $\Gamma$ is 
\[
k=(2m-1) e - 2m v +mD.  
\]
We furthermore define the \emph{weight} of $\Gamma$ as 
\[
W=2(e-v) +D.
\]
The weight is preserved by the differential, and hence the graph complexes split into a direct sum of pieces of fixed weight.

Next, note that all three graph complexes are equipped with a commutative product by taking disjoint union of graphs. This product is compatible with the differential and the weight grading, since the weight is additive.
Furthermore, any graph splits uniquely as a union of connected graphs. It follows that our three graph complexes above are quasi-free as (non-unital) graded commutative algebras, so that we can write 
\begin{align*}
  \fG_{(g),1}&\cong S^+(\G_{(g),1}[-1]) &
  \fG_{(g),1}^\tp&\cong S^+(\G_{(g),1}^\tp[-1]) &
  \fG_{(g)}&\cong S^+(\G_{(g)}[-1]).
\end{align*}
Here we denoted by $\G_{(g),1}[-1]\subset \fG_{(g),1}$ the subcomplex spanned by connected graphs, and similarly for the other complexes.
Since the differential is compatible with the commutative algebra structure (i.e., a derivation) in each case, it endows the degree shifted connected parts $\G_{(g),1}$, $\G_{(g),1}^\tp$ and $\G_{(g)}$ with dg Lie coalgebra structures.
We shall denote the dual dg Lie algebras by 
\begin{align*}
  \GC_{(g),1}&:=\G_{(g),1}^* &
  \GC_{(g),1}^\tp&:= (\G_{(g),1}^\tp)^* &
  \GC_{(g)}&:= \G_{(g)}^*
\end{align*}
and the differential by
\[
\delta = d^* = d_c^* + d_{cut}^* + (d_{mul}'')^* = \delta_{split} + \delta_{glue} + \delta_Z.
\]
The combinatorial expressions for the different pieces of the (adjoint) differential are given in the introduction. We note in particular that we absorb the sign into the definition of $\delta_{split}$ for our convenience.

\subsection{Action of the orthogonal or symplectic group}
All graph complexes considered carry a natural action of the orthogonal or symplectic group $\OSp_g$ of linear endomorphisms of $V_g:= H^m(W_{g,1})= H^m(W_{g})$ that preserve the natural pairing on $V_g$.
The action is by transforming the decorations in $\bar H(W_{g,1})$.
Accordingly, one has an action of the Lie algebra of  $\OSp_g$, that we denote by $\osp_g^0$.

In the case of the graphs complexes $\fG_{(g)}$, $\G_{(g)}$ and $\GC_{(g)}$ the action of $\osp_g^0$ may be enlarged to a (non-flat, see below) action of a graded Lie algebra 
\[
\osp_g := \osp_g^0 \ltimes \osp_g^{nil},  
\]
of graded endomorphisms of $H(W_g)$ that preserve the pairing and kill the unit element $1\in H^0(W_g)$.
To be specific, any element $\osp_g^{nil}$ is an  endomorphism
\begin{gather*}
  \phi_c : H(W_g) \to H(W_g)
\end{gather*}
defined for $c\in H^m(W_g)$ such that 
\begin{align*}
\phi_c(\omega) &= c
\\
\phi_c(x) &= -\langle c,x\rangle \, 1 \quad\text{for $x\in H^m(W_g)$} \\
\phi_c(1) &= 0.
\end{align*}
(Here $\omega$ is again the cohomology class represented by the unit volume form on $W_g$.)
The endomorphisms $\phi_c$ preserve the pairing in the sense that 
\[
\langle \phi_c(x),y\rangle +  (-1)^{m|x|} \langle x,\phi_c(y)\rangle=0
\]
for any homogeneous $x,y\in H(W_g)$.
As an $\OSp_g$- (and $\osp_g$-)representation $\osp_g^{nil}$ is equivalent to the $2g$-dimensional defining representation, concentrated in cohomological degree $-m$.
Furthermore, one checks that $\osp_g^{nil}\subset \osp_g$ is an abelian Lie subalgebra, i.e, the commutators $[\phi_c,\phi_d]=0$ vanish.

We now extend the dg Lie algebra $\GC_{(g)}$ to the dg Lie algebra $\GCex_{(g)}$, defined as a graded vector space as 
\[
\GCex_{(g)} := \osp_g^{nil} \oplus \GC_{(g)}.
\]
We endow $\GCex_{(g)}$ with a graded Lie algebra structure by extending the Lie bracket on $\GC_{(g)}$ to a Lie bracket on $\GCex_{(g)}$ such that for $\phi_c\in \osp_{g}^{nil}$ and $\Gamma\in \GC_{(g)}$ 
$
  [\phi_c,\Gamma]  
$
is obtained by acting with $\phi_c$ on the decorations in $\bar H(W_g)$. We extend the differential such that 
\[
  \delta \phi_c 
  =\frac 12 
  \begin{tikzpicture}
    \node[int,label=90:{$c^* \Delta_{1}^*$}] (v) at (0,0){};
  \end{tikzpicture},
\]
with $c^*\in H_m(W_g)$ the Poincar\'e dual element to $c$, and $\Delta_1^*\in H_m(W_g)\otimes H_m(W_g)$ the diagonal element in the homology dual to the diagonal element in cohomology $\Delta_1$ as in \eqref{equ:Delta explicit}. (We will later also use the notation $\Delta_1$ (abusively) for the diagonal in homology instead of $\Delta_1^*$ to simplify notation.)
One may check that with this differential and Lie bracket $\GCex_{(g)}$ is a well-defined dg Lie algebra.
We may also extend the weight grading on $\GCex_{(g)}$ by declaring the elements of $\osp_g^{nil}$ to have weight 1. Note that this is necessary so that the differential and Lie bracket still preserve the weight.

For notational convenience we will define for $\alpha\in H_m(W_g)$ 
\[
\phi_\alpha := \phi_{\alpha^*}  
\]
using the identification of $H_m(W_g)$ and $H^m(W_g)$ by Poincar\'e duality.

Finally, we also introduce the following notation: We denote the graded dual Lie coalgebra by
\[
\Gex_{(g)} := (\GCex_{(g)})^c.
\]
As a graded vector space this is $\G_{(g)}\oplus \osp_g^{nil,c}$.
Furthermore, we consider the version with disconnected graphs $\fGex_{(g)}$ defined as the cobar construction (Chevalley-Eilenberg complex)
\[
  \fGex_{(g)} = \Bar^c \Gex_{(g)}.
\] 
This complex also has a graphical interpretation, see section \ref{sec:CE interpretation} below.

\subsection{Auxiliary graph complexes $\Graphsg(n)$}

In \cite{CamposWillwacher} Campos and the third author define dg graphical cocommutative coalgebra models
for framed configuration spaces of $n$ points on surfaces (i.e., for $W_g$ if $m=1$ in our notation).
As auxiliary objects in some proofs below, we will need to use closely analogous graph complexes defined for general $m$, which we denote by $\Graphsg(n)$. Elements in $\Graphsg(n)$ are linear combinations of graphs as in $\GCg$, but with $n$ additional labelled vertices (called ``external vertices'', labelled by $\{1, \dots, n\}$), and for which tadpoles are allowed only at external vertices. Additionally, we require that every connected component contains at least one external vertex. 
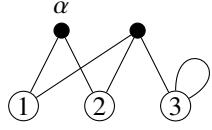
\[
\begin{tikzpicture}
  \node[int,label=90:{$\alpha$}] (a) at (0.5,1) {};
  \node[int] (b) at (1.5,1) {};
  \node[ext] (u) at (0,0) {$1$};
  \node[ext] (v) at (1,0) {$2$};
  \node[ext] (w) at (2,0) {$3$};
  \draw (u) edge (a);
  \draw (v) edge (a);
  \draw (b) edge (u);
  \draw (b) edge (v);
  \draw (b) edge (w);
  \draw (w) edge [out=120-30, in=60-30, loop, looseness=10] (w);
\end{tikzpicture}
\]
The differential is as in $\GCg$, where additionally the splitting operation $\delta_{split}$ (see \eqref{equ:deltasplit}) splits an external vertex into an internal and an external vertex connected by an edge. The coalgebra structure is given by decomposing a graph into internally-connected components (graphs that stay connected after deleting the external vertices). More precisely, $\Graphsg(n)$ is freely cogenerated by $\icgg(n)$, the space of internally connected graphs, that is
\[
\Graphsg(n) = \Bar (\icgg(n)).
\]
Furthermore, $\GCex_{(g)}$ naturally acts on $\Graphsg(n)$ via the ``glueing''-operation by gluing decorations from a graph in $\GCex_{(g)}$ to decorations on a graph in $\Graphsg(n)$.

For what follows, we will only need $\Graphsg(1)$, that is graphs with one external vertex. There is a natural quasi-isomorphism of coalgebras
\begin{equation}\label{eq:BVGraphsg to CWg}
\Graphsg(1) \leftarrow (H_\bullet(W_{g}) \oplus \theta H_\bullet(W_{g}), d) =: C_\bullet(W_{g})
\end{equation}
determined on cogenerators by mapping
\begin{itemize}
    \item $\alpha \in H_\bullet(W_{g})$ to the graph with a single decoration $\alpha$ at the external vertex,
    \item $\theta 1$ to the tadpole,
    \item $\theta c_i$ and $\theta \omega$ to $0$.
\end{itemize}
The only non-zero differential on the right-hand side is $d\omega = (2+(-1)^m 2g)\theta$. The morphism above is dual to the usual projection onto graphs with no internal vertices and multiplying decorations which  commonly appears in the literature (\cite[Section 6]{idrissi}, \cite[Appendix A]{CamposWillwacher}), and it is known to be a quasi-isomorphism.

We also note that the complex $C_\bullet(W_{g})$ computes the homology of the sphere bundle of the tangent bundle of $W_g$. 

\begin{lemma}
The map \eqref{eq:BVGraphsg to CWg} is a quasi-isomorphism.
\end{lemma}
\begin{proof}
First, filter by the total degree of the decorations in $\bar H_\bullet(W_g)$. The associated graded complex then splits as the direct sum of subcomplexes of graphs with a tadpole at the external vertex and graphs without. Removing the tadpole defines an isomorphism of the two complexes, thus it remains to see that the tadpole-free complex $\Graphsg^{\notadp}(1)$ is quasi-isomorphic to $H_\bullet(W_g)$. Finally, we write $\Graphsg^\notadp(1)$ as a the total complex of the two-term complex
\[
{\Graphsg^{\notadp,\geq 2}(1)} \to \Graphsg^{\notadp, \leq 1 }(1),
\]
where ${\Graphsg^{\notadp,\geq 2}(1)}$ and ${\Graphsg^{\notadp, \leq 1}(1)}$ are graphs for which the external vertex has valency at least two, and no more than one, respectively. The map ${\Graphsg^{\notadp,\geq 2}(1)} \to {\Graphsg^{\notadp, \leq 1}(1)}$ is injective and has cokernel represented by graphs without any edges and at most one decoration at the external vertex which is identified with $H_\bullet(W_g)$.
\end{proof}

Finally, let us also mention the following version corresponding to $W_{g,1}$. Namely, we can consider the dg subcoalgebra $\Graphs^\notadp_{(g),1}(1) \subset \Graphsg(1)$ consisting of graphs that do not contain any decorations with the top homology class $\omega$ and no tadpoles. In this case, the action above induces an action of $\GC_{(g),1}$. Moreover, $\Graphs^\notadp_{(g),1}(1)$ is a coalgebra model for $H(W_{g,1})$ by the same proof as above.

\section{Proof of part (i) of Theorem \ref{thm:GC comparison}}\label{sec:part i proof}
As defined above $\GC_{(g),1}=\GC_{(g),1}^\tp/I_g^\tp$, with $I_g^\tp$ being the (closed) subspace spanned by graphs with at least one tadpole.
We will check that $H(I_g^\tp)$ is $2g$-dimensional, concentrated in degree $2-m$, spanned by graphs of the form 
\begin{equation}\label{equ:Gamma a}
  \Gamma_\alpha =
\begin{tikzpicture}
  \node[int,label=-90:{$\alpha$}] (v) at (0,0) {};
  \draw (v) edge[loop] (v);
\end{tikzpicture},
\end{equation}
with $\alpha\in H^m(W_{g,1})$.
The graphs $\Gamma_\alpha$ have weight $1$, hence $H(I_g^\tp)$ is zero in weights $\geq 2$, and part (i) of Theorem \ref{thm:GC comparison} follows.

To compute $H(I_g^\tp)$ we proceed analogously to the proof of \cite[Proposition 3.4]{Willgrt}, but we shall streamline the proof a bit.
First we filter $I_g^\tp$ by the number of loops, i.e., we define the bounded above, descending complete filtration
\[
  I_g^\tp =  \mF^1 I_g^\tp \supset \mF^2 I_g^\tp \supset \cdots
\] 
with $\mF^p I_g^\tp\subset I_g^\tp$ spanned by graphs with first Betti number at least $p$.
The (complete) associated graded complex may be identified with $(I_g^\tp,\delta_{split})$.
We claim that $H(I_g^\tp,\delta_{split})$ is $2g$-dimensional, concentrated in degree $2-m$, and spanned by the classes of the graphs $\Gamma_\alpha$ above.
If we assume this claim the proof concludes, since the spectral sequence must converge at this point by degree reasons.

To show the claim, let us call a \emph{long tadpole} a tadpole at a trivalent vertex attached to an edge (i.e., not to a decoration)
\[
\begin{tikzpicture}
  \node[int] (v) at (0,0) {};
  \draw (v) edge[loop] (v) edge +(0,-.7);
\end{tikzpicture}
\]
and the edge adjacent to the tadpole the long tadpole edge.
We filter $(I_g^\tp,\delta_{split})$ again by the number of edges that are \emph{not} long tadpole edges. The associated graded complex may then be identified with $(I_g^\tp,\delta_{split}')$, with $\delta_{split}'$ the piece of the differential that creates one long tadpole edge,
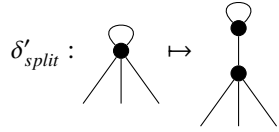
\[
  \delta_{split}' : 
\begin{tikzpicture}
  \node[int] (v) at (0,0) {};
  \draw (v) edge[loop] (v) edge +(0,-.7) edge +(.5,-.7) edge +(-.5,-.7);
\end{tikzpicture}
\mapsto 
\begin{tikzpicture}
  \node[int] (v) at (0,.3) {};
  \node[int] (w) at (0,-.3) {};
  \draw (v) edge[loop] (v) edge (w) (w) edge +(0,-.7) edge +(.5,-.7) edge +(-.5,-.7);
\end{tikzpicture}
\]

This differential has an obvious homotopy
\begin{gather*}
h :  I_g^\tp \to I_g^\tp \\
I_g^\tp\ni \Gamma\mapsto h\Gamma = \sum_{e} (-1)^{e+1} \Gamma/e,
\end{gather*}
where the sum is over all long tadpole edges, and $\Gamma/e$ is obtained by contracting the edge.
One easily computes that 
\[
(h \delta_{split}'+\delta_{split}'h )\Gamma 
=
N(\Gamma) \Gamma,
\]
with $N(\Gamma)$ the number of tadpoles in the graph $\Gamma$ that are either long tadpoles, or attached to a vertex of valency $\geq 4$ (including the two tadpole half-edges).  
But since there must be at least one tadpole in the graph by definition on $I_g^\tp$, we see that $N(\Gamma)>0$ for all graphs except for those of the form $\Gamma_\alpha$, thus finishing the computation of $H(I_g^\tp)$, an thus the proof that \eqref{equ:thm GC comparison 1} is an isomorphism for $W\geq 2$.

\subsection{Comparison in weight 1}\label{sec:weight 1 comparison}

Note that part (i) of Theorem \ref{thm:GC comparison} excludes weight 1. In this section we shall fill this gap and compute the weight 1 parts of our three dg Lie algebras.

First consider $\gr^1\GC_{(g),1}$.
This complex is spanned by graphs with a single vertex and three decorations 
\begin{equation}\label{equ:Gama abc}
  \Gamma_{\alpha\beta\gamma}:=
    \begin{tikzpicture}
      \node[int,label=90:{$\alpha\beta\gamma$}] {};
    \end{tikzpicture}
    \quad \text{with $\alpha,\beta,\gamma\in H_m(W_{g,1})$},
\end{equation}
living in cohomological degree $1-m$.
The differential is zero, and hence the cohomology is trivially identified with 
\[
  H(\gr^1\GC_{(g),1}) = \gr^1\GC_{(g),1} \cong
  S^3 \bar H(W_{(g),1})[-1-2m].
\]

Next consider $\gr^1\GC_{(g),1}^\tp$. This complex is spanned by the graphs $\Gamma_{\alpha\beta\gamma}$ of \eqref{equ:Gama abc}, and in addition the graphs $\Gamma_\alpha$ of \eqref{equ:Gamma a}.
The differential acts as 
\[
\delta \Gamma_{\alpha\beta\gamma}
= 
\delta_{glue} \Gamma_{\alpha\beta\gamma}
= 
\langle \alpha, \beta \rangle \Gamma_\gamma
+
\langle \beta, \gamma \rangle \Gamma_\alpha
+
\langle \gamma, \alpha \rangle \Gamma_\beta,
\]
with $\langle-,-\rangle$ the canonical pairing on $\bar H^m(W_{g,1})$.
Note that the graphs $\Gamma_\alpha$ span a $2g$-dimensional subspace equivalent to the defining ($2g$-dimensional) representation of $\OSp_g$. Since this representation is irreducible, we know by Schur's Lemma that $\delta$ must be either surjective or zero.
If $g\geq 2$ one easily checks that $\delta$ is not zero, and hence surjective.
For $g=0$ one has that trivially $\gr^1\GC_{(g),1}=0$.
For $g=1$ and $m$ odd one has $S^3 \bar H(W_{(g),1})=0$, and hence $\delta$ is zero trivially.
For $g=1$ and $m$ even one again checks easily that $\delta$ is not zero and hence injective.

Finally we turn to $\gr^1\GCex_{(g)}$. 
This complex is spanned by the graphs $\Gamma_{\alpha\beta\gamma}$ for $\alpha,\beta,\gamma\in H_m(W_g)$, and the elements $\phi_{\alpha} \in \osp_g^{nil}$.
The differential is define such that 
\[
\delta \phi_\alpha = \frac 1 2 \sum_{ij} g_{ij} \Gamma_{c_ic_j\alpha},
\]
with $\Delta_1^*=\sum_{i,j} g_{ij} c_i\otimes c_j\in H_m(W_g)\otimes H_m(W_g)$ a representation of the diagonal element in some basis $\{c_i\}$ of the homology.
The space $\osp_g^{nil}$ is equivalent to the defining $\OSp_g$-representation, and hence the differential is either zero or injective by Schur's Lemma.
If $g\geq 2$ the differential is non-zero as one checks on an element, and hence injective. For $g=0$ we have $\gr^1\GCex_{(g)}=0$ trivially. 
For $g=1$ and $m$ odd all $\Gamma_{\alpha\beta\gamma}$ are zero and hence the differential is zero.
For $g=1$ and $m$ even the differential is again injective as one checks by explicit computation.

Overall we have the following table of cohomologies for $m$ odd, as $\Sp(2g)$-representations 
\begin{center}
\begin{tabular}{l|cccc}
& $g=0$ & $g=1$ & $g=2$ & $g\geq 3$  \\
\hline 
$\gr^1H(\GC_{(g),1}^\tp)$
 & 0 & $V(\lambda_1)[m-2]$ & 0 & $V(\lambda_3)[m-1]$  
\\ $\gr^1H(\GC_{(g),1})$
& 0 & 0 & $V(\lambda_1)[m-1]$ & $(V(\lambda_1)\oplus V(\lambda_3))[m-1]$
\\ $\gr^1H(\GCex_{(g)})$
 & 0 & $V(\lambda_1)[m]$ & 0 & $V(\lambda_3)[m-1]$  
\end{tabular}
\end{center}
Here $V(\lambda)$ denotes a representation of highest weight $\lambda$, and $\lambda_1, \lambda_2,\dots$ is a system of fundamental weights.

For $m$ even we similarly have the following table.
\begin{center}
  \begin{tabular}{l|cc}
    & $g=0$ & $g\geq 1$  \\
\hline 
$\gr^1H(\GC_{(g),1}^\tp)$
 & 0 & $V(3\lambda_1)[m-1]$ 
\\ $\gr^1H(\GC_{(g),1})$
 & 0 & $(V(\lambda_1) \oplus V(3\lambda_1))[m-1]$
\\ $\gr^1H(\GCex_{(g)})$
 & 0 & $V(3\lambda_1)[m-1]$
  \end{tabular}
\end{center}
Here one shall remark that providing highest weights only describes a representation of $\mathit{SO}(g,g)$ a priori, but the representations here extend to $O(g,g)$.
Concretely, $V(\lambda_1)\cong V_g$ shall be the defining $2g$-dimensional representation of $O(g,g)$ and $V(3\lambda_1)$ the kernel of the canonical map 
\[
S^3 V(\lambda_1) \to V(\lambda_1)
\]
defined by contraction with the pairing.

In all cases the canonical maps $\GC_{(g),1}^\tp\to \GC_{(g),1}\to \GCex_{(g)}$ send the summands $V(\lambda_3)$ (resp. $V(3\lambda_1)$) isomorphically onto themselves, as the corresponding cohomology classes are always represented by graphs of the form $\Gamma_{\alpha\beta\gamma}$ that exist in all complexes.

\section{Comparing $\GC_{(g),1}$ and $\GC_{(g)}$, and the proof of part (ii) of Theorem \ref{thm:GC comparison}}

The action of $\GCex_{(g)}$ on $\Graphsg(1)$ by (co)derivations yields a morphism of dg Lie algebras
\[
\GCex_{(g)}=\osp_g^{nil}\ltimes \GC_{(g)} \rightarrow \mathrm{Der}(\Graphsg(1)).
\]
We consider the morphism induced by the action on the coaugmentation
\[
f:\GCex_{(g)}  \rightarrow \icgg(1)[1].
\]
The map is given as follows:
\begin{itemize}
    \item Elements $\Gamma \in \GC_g$ act on the ``empty'' graph with one external vertex and no internal vertices or edges. Namely, we sum over all ways of connecting an $\omega$-decoration to the external vertex.
    \item If $\phi_{c_i}\in \osp_g^{nil}$ is an element that sends $1\in H_0(W_g)$ to $c_i \in H_m(W_g)$ then we map it to the graph with one external vertex and decoration $c_i$.
\end{itemize}
Note that hence elements of the dg Lie subalgebra $\GC_{(g),1}$ preserve the coaugmentation, since none of the vertices carry $\omega$-decorations. 


\begin{prop}
The total complex of the short exact sequence
\[
\GC_{(g),1} \rightarrow \GCex_{(g)} \xrightarrow{f} \icgg(1)[1]
\]
is acyclic.
\end{prop}

\begin{proof}
We show the equivalent claim that the inclusion
\[
\GC_{(g),1} \rightarrow  \textrm{cone}(f)=\GCex_{(g)} \oplus \icgg(1)[1][-1]
\]
is a quasi-isomorphism. Denote by $\icgg^\bullet(1) \subset \icgg(1)$ the subcomplex consisting of diagrams having at least one internal vertex, and by $\icgg^0(1)$ the subspace spanned by diagrams with no internal vertices. The latter does not define a subcomplex. Notice moreover that $f(\GCg)\subset \icgg^\bullet(1)[1]$, while $f(\osp^{nil}_g)\subset \icgg^0(1)[1]$. Next, we equip $\GC_{(g),1}$ with the complete descending filtration given by
\[
\# \text{ internal vertices } + \# \text{ edges } \geq p.
\]
and similarly the mapping cone with the complete descending filtration
\[
\# \text{ internal vertices } + \# \text{ internal edges } \geq p
\]
where by an internal edge we mean an edge connecting two internal vertices. On the level of the associated graded complexes we have, on the one hand, $(\gr \GC_{g,1},0)$, whereas $\gr \mathrm{cone}(f)$ splits into the direct sum of subcomplexes
\[
\gr \mathrm{cone}(f)\cong \gr(\osp^{nil}_g \oplus \icgg^0(1)) \oplus \gr (\GCg \oplus \icgg^\bullet(1)).
\]
On the first subcomplex the differential has two parts. The internal differential on $\icgg^0 (1)$ which maps
\[
\begin{tikzpicture}
  \node[ext,label={$\omega$}] (v) at (0,0) {$1$};
\end{tikzpicture}
 \ \ \mapsto
\begin{tikzpicture}
  \node[ext] (v) at (0,0) {$1$};
  \draw (v) edge [out=120, in=60, loop, looseness=10] (v);
\end{tikzpicture}
\]
and the one induced by $f$ given by
\[
\osp_g^{nil}\ni(\phi_{c_i}:1\mapsto c_i) \mapsto
\begin{tikzpicture}
  \node[ext,label={$c_i$}] (v) at (0,0) {$1$};
\end{tikzpicture}
\]
As such the subcomplex $\gr (\osp^{nil}_g \oplus \icgg^0(1))$ is acyclic. On the other hand, notice that on $ \gr (\GCg \oplus \icgg^\bullet(1))$ the differential is defined by replacing $\omega$-decorations with an edge connected to the external vertex.

\[
\gr d:
\begin{tikzpicture}
  \node[int,label=0:{$\omega$}] (v) at (0,0) {};
  \node[ext] (w) at (1,-0.5) {$1$};
  \draw (v) edge +(-.5,.5) edge +(.5,.5) edge +(0,.5);
  \draw (w) edge +(.5,.5) edge +(0,.5);
\end{tikzpicture}
\ \mapsto \ 
\begin{tikzpicture}
  \node[int] (v) at (0,0) {};
  \node[ext] (w) at (1,-0.5) {$1$};
  \draw (v) edge +(-.5,.5) edge +(.5,.5) edge +(0,.5);
  \draw (w) edge +(.5,.5) edge +(0,.5);
  \draw (v) edge (w);
\end{tikzpicture}
\]
This differential has an obvious homotopy
\[
h:\GCg \oplus \icgg^\bullet(1)\rightarrow \GCg \oplus \icgg^\bullet(1)
\]
given by replacing edges connected to the external vertex by an $\omega$-decoration at the respective internal vertex. It satisfies 
\[
(\gr d h+h \gr d)(\Gamma)=N(\Gamma) \Gamma
\]
where $N(\Gamma)$ is the number of $\omega$-decorations plus the number of edges incident at the external vertex. In particular this means that the cohomology is spanned by diagrams with no $\omega$-decorations and no edges connected to the external vertex, i.e. diagrams in $\GC_{(g),1}\subset \GCg\oplus \icgg^\bullet(1)$.
\end{proof}

We obtain a long exact sequence on cohomology groups
\begin{equation}\label{equ:GC long exact icg}
  \cdots \to H^k(\GC_{(g),1}) \to  H^k(\GCex_{(g)} )\xrightarrow{H(f)} H^{k+1}(\icgg(1)) \to H^{k+1}(\GC_{(g),1}) \to \cdots.
\end{equation}

\begin{lemma}\label{lem:center icg}
The image of
\[
H(f):H^{\bullet}(\GCex_{(g)})\rightarrow H^{\bullet+1}(\icgg(1))
\]
lies in the center of the graded Lie algebra $H(\icgg(1))$.
\end{lemma}

\begin{proof}
Let $\gamma \in \GCg$ and $\Gamma \in \icgg(1)$ be closed elements. Consider the disjoint union
\[
\gamma \sqcup \Gamma \in \fGraphs_{(g)}(1)
\]
where $\fGraphs_{(g)}(1)$ is defined analogously to $\Graphs_{(g)}$, except that we drop the connectivity condition. Take the projection $\pi:\Graphsg(1)\rightarrow \icgg(1)[1]$ onto the internally connected part of the equation $0=d^2(\gamma \sqcup \Gamma)$. We obtain
\[
0=\pi d^2(\gamma \sqcup \Gamma)=d_{\icgg} (\pi(\gamma \cdot \Gamma))+ [f(\gamma),\Gamma].
\]
Notice, that $\gamma\cdot \Gamma$ lies in $\Graphsg(1)$, and we again project onto its internally connected part before applying the differential. In particular, $[f(\gamma),\Gamma]=0 \in H(\icgg(1))$ for all closed $\Gamma$, and thus $f(\gamma)$ lies in the center of $H(\icgg(1))$.
\end{proof}

\newcommand{\wg}{\mathfrak w}
\newcommand{\hwg}{\hat{\mathfrak w}}
\subsection{Rational homotopy Lie algebra of $W_{g}$}

We define the graded Lie algebra $\wg_g$ to be the Koszul dual Lie algebra to the graded commutative algebra $H^\bullet(W_g)$.
Concretely, $\wg_g$ has the following presentation. It is generated as a graded Lie algebra by symbols $c_1,\dots, c_{2g}$ of degree $1-m$, corresponding to a basis of $H_m(W_g)$. Let us denote the matrix of the non-degenerate pairing with respect to the given basis by $g_{ij}$.
Then 
\[
  \wg_g = \FreeLie(c_1,\dots,c_{2g}) / (\sum_{i,j=1}^{2g} g_{ij} [c_i,c_j] =0).
\]
It is clear that $\wg_g$ is the quadratic dual to the quadratic commutative algebra $H^\bullet(W_g)$. To see that it is the Koszul dual it suffices to show that either $H^\bullet(W_g)$ or $U(\wg_g)$ is a Koszul algebra. Note that $U(\wg_g)$ is a quadratic algebra with generators $c_i$ and the single relation 
\[
\sum_{i,j=1}^{2g} g_{ij} [c_i,c_j] =0
\]
and it is Koszul by \cite[Theorem 4.1.1]{lodayval} since in particular there are no critical monomials on which to test the confluence condition.

Similarly, we also consider a central extension of this graded Lie algebra with a central element $c$ of degree $2-2m$.
\[
  \wg_g^{fr} = \FreeLie(c, c_1,\dots,c_{2g}) / (\sum_{i,j=1}^{2g} g_{ij} [c_i,c_j] - (2+(-1)^m 2g) c =[c,c_1]=\cdots = [c,c_{2g}]=0).
\]
This is a Quillen model for the sphere bundle $STW_g$. One obviously has a map of graded Lie algebras
\[
  \wg_g^{fr} \to \wg_g  
\]
by setting $c=0$.
Both Lie algebras above carry an additional grading, assigning the generators $c_i$ degree $+1$ and the generator $c$ degree $+2$.
We denote by $\hwg_g^{fr}$ and $\hwg_g$ the respective completions.
(They only differ from $\wg_g^{fr}$ and $\wg_g$ in the case $m=1$, in fact.)

We shall need the following result, which can be deduced from \cite[Proposition A']{AsadaKaneko} (see also \cite[Lemma 74]{Matteo} for a different proof).
\begin{prop}[{\cite[Proposition A']{AsadaKaneko}}]\label{prop:center wg}
For $g\geq 2$ the center of $\wg_g$ is trivial, and the center of $\wg_g^{fr}$ is one-dimensional, spanned by $c$.
\end{prop}

There is a morphism of $L_\infty$-algebras
\[
\phi:\icgg(1) \rightarrow \hwg_g^{fr}
\]
which is given as follows.
\begin{itemize}
    \item If a graph $\Gamma \in \icgg(1)$ contains either a loop, a $\geq 4$-valent vertex or an $\omega$-decoration, then $\phi(\Gamma) = 0$.
    \item If $\Gamma \in \icgg(1)$ is a trivalent tree with no $\omega$-decorations, it can be interpreted as a Lie tree and determines an element in the free Lie algebra $L_\Gamma \in \FreeLie(c_1, \dots, c_{2g}) \to \hwg_g^{fr}$. In that case $\phi(\Gamma) = [L_\Gamma] \in \hwg_g^{fr}$.
\end{itemize}

\begin{lemma}\label{lem:icgwg}
The morphism $\phi:\icgg(1) \rightarrow \hwg_g^{fr}$ defines a quasi-isomorphism.
\end{lemma}

\begin{proof}
By bar-cobar duality it suffices to show that the induced morphism on the respective bar complexes 
\[
\Graphsg(1)\rightarrow \Bar (\hwg_g^{fr})
\]
is a quasi-isomorphism. To see this, first notice that the morphism on the left of the composition
\[
C_\bullet(W_{g}) \rightarrow \Graphsg(1)\rightarrow \Bar (\hwg_g^{fr})
\]
is the quasi-isomorphism \eqref{eq:BVGraphsg to CWg}. In order to prove the claim, it now suffices to show that the composition $ C_\bullet(W_{g})\to \Bar (\hwg_g^{fr}) $ is also a quasi-isomorphism. Indeed, notice that on the second page of the Hochschild-Lyndon-Serre spectral sequence, we find $H(\Bar(\hwg_g))\otimes H(\Bar(\langle c\rangle))$, where $\langle c \rangle$ denotes one-dimensional Lie algebra generated by $c$. The tensor product $H(\Bar(\hwg_g))\otimes H(\Bar(\langle c\rangle))$ is identified with $C_\bullet(W_{g})$, where we used that $H(\Bar(\hwg_g))=H(W_g)$.
\end{proof}

\begin{prop}\label{prop:long exact splits}
For $g \geq 2$, we have $H(\phi)=0$. In particular, the long exact sequence \eqref{equ:GC long exact icg} splits into
\begin{equation}\label{equ:GC split long exact icg}
  0\rightarrow \hwg_g^{fr} \to H(\GC_{(g),1}) \to  H(\GCex_{(g)} )\rightarrow 0.
\end{equation}
For $g = 1$, the sequence splits into
\begin{gather*}
  0\rightarrow \K c \to H(\GC_{(1),1}) \to  H(\GC^{\mathrm{ex}}_{(1)} )\rightarrow \K c_1 \oplus \K c_2 \rightarrow 0, \quad\text{if $m$ is odd,}\\
0\rightarrow \K c \oplus \K c_1 \oplus \K c_2 \to H(\GC_{(1),1}) \to  H(\GC^{\mathrm{ex}}_{(1)} )\rightarrow \K [c_1,c_1] \oplus \K [c_2,c_2] \rightarrow 0 , \quad\text{if $m$ is even.}
\end{gather*}
\end{prop}

\begin{proof} Let us first treat the case $g \geq 2$. By Lemma \ref{lem:center icg} and Proposition \ref{prop:center wg} it suffices to show that the connecting homomorphism $g:\hwg_g^{fr}\rightarrow H(\GC_{(g),1})$ is injective on the center $Z(\hwg_g^{fr})=\K c$. The central element $c$ has the following representative in $H(\icgg(1))$
\[
\begin{tikzpicture}
  \node[int,label={$\Delta_1$}] (v) at (0,0) {};
  \node[ext] (w) at (0,-.5) {$1$};
  \draw (v) edge (w);
\end{tikzpicture}
\]
where $\Delta_1 \in H_m(W_g)^{\otimes 2}$ is the diagonal element. The connecting homomorphism maps this element to
\[
\begin{tikzpicture}
  \node[int,label=left:{$\Delta_1$}] (v) at (0,0) {};
  \node[int,label=right:{$\Delta_1$}] (w) at (0.5,0) {};
  \draw (v) edge (w);
\end{tikzpicture}
\in \GC_{(g),1}.
\]
Consider the natural map 
\begin{equation}\label{eq:GCg1 to Freelie}
H(\GC_{(g),1})\rightarrow \mathrm{Der}(\FreeLie(\bar H(W_{g,1})))
\end{equation}
given by projecting onto the trivalent tree part before performing the glueing operation onto any Lie tree. The element above is sent to the inner derivation determined by
\[
2 \sum_{i,j=1}^{2g} g_{ij} [c_i,c_j]
\]
and is hence non-zero in $H(\GC_{(g),1})$. Note that the last statement also holds for $g=1$. The argument works similarly for the other cases in $g=1$. In this case, $\hwg_g^{fr}$ is finite-dimensional and given by
\[
\hwg_g^{fr}=
\begin{cases}
  \K c \oplus \K c_1 \oplus \K c_2 & \text{ for } m \text{ odd, }\\
  \K c \oplus \K c_1 \oplus \K c_2 \oplus \K [c_1,c_1] \oplus \K [c_2,c_2] & \text{ for } m \text{ even. }
\end{cases}
\]
For $m$ odd, $\osp^{nil}_g$ is two-dimensional, both generators are closed in $\GCex_{(g)}$ and their image is precisely $\K c_1 \oplus \K c_2$. For $m$ even, $H(f)$ maps the closed graphs
\[
\begin{tikzcd}
\node[int,label=90:{c_1 c_1 \omega}] {};
\end{tikzcd}
\text{ and }
\begin{tikzcd}
\node[int,label=90:{c_2 c_2 \omega}] {};
\end{tikzcd}
\in \GC_{(g)}
\]
to $[c_1,c_1]$ and $[c_2,c_2]$, respectively, while the connecting homomorphism sends $c_1$ and $c_2$ to 
\[
  \begin{tikzpicture}
    \node[int,label=90:{$c_1 c_1 c_2$}] {};
  \end{tikzpicture}
\text{ and }
    \begin{tikzpicture}
    \node[int,label=90:{$c_2 c_2 c_1$}] {};
  \end{tikzpicture}
\]
respectively. Under the natural map \eqref{eq:GCg1 to Freelie} the latter are sent to the derivations of $\FreeLie(\bar H(W_{1,1}))$
\begin{align*}
c_1 &\mapsto 
\begin{cases}
  c_1 & \mapsto [c_1,c_1]\\
  c_2 & \mapsto 2[c_1,c_2].
\end{cases} & c_2 &\mapsto 
\begin{cases}
  c_1 & \mapsto 2[c_1,c_2]\\
  c_2 & \mapsto [c_2,c_2]
\end{cases} 
\end{align*}
which shows that $\K c \oplus \K c_1 \oplus \K c_2$ maps injectively into $H(\GC_{(1),1})$.
\end{proof}

The connecting homomorphism sends generators $c_j$ of $\hwg_g^{fr}$ to 
\[
  \begin{tikzcd}
    \node[int,label=90:{c_j\Delta_1}] {};
  \end{tikzcd}
  \in \GC_{(g),1},
\]
with $\Delta_1\in H_m(W_{g,1})^{\otimes 2}$ being the diagonal element. This completely determines it by the following. 
\begin{prop}
The connecting homomorphism $\hwg_g^{fr} \rightarrow H(\GC_{(g),1})$ is a homomorphism of Lie algebras.
\end{prop}

\begin{proof}
The statement is equivalent to showing that the identification of the cocone of $\GC_{(g),1}\rightarrow \GCex_{(g)}$ with $\hwg_g^{fr}$ given in Proposition \ref{prop:long exact splits} can be made into a quasi-isomorphism of dg Lie algebras. Notice that cocones of dg Lie algebras are again dg Lie algebras and the natural projection (which is the connecting homomorphism from above) is a map of dg Lie algebras.
We have the following commuting diagram of dg Lie algebras
\[
\begin{tikzcd}
  \GC_{(g),1} \ar[r]\ar[d] & \GCex_{(g)} \ar[d]\\
  \mathrm{Der}_+(\Graphsg(1)) \ar[r]\ar[d] & \mathrm{Der}(\Graphsg(1))  \ar[d] \\
  \mathrm{Der}(\Bar^c (\Graphsg (1))) \ar[r] & \Bar^c (\Graphsg(1))[1] \rtimes \mathrm{Der}(\Bar^c (\Graphsg(1)))
\end{tikzcd}
\]
where $\mathrm{Der}_+ (\Graphsg(1))$ are derivations preserving the coaugmentation and the morphisms going from the second to third row are as in \cite[Theorem 1]{Gatsinzi}. It induces quasi-isomorphims of dg Lie algebras between the horizontal cocones. To see this, we take the cone of the morphism of the first row and the cokernel of the two other, to obtain the sequence of quasi-isomorphisms of dg vector spaces
\[
\mathrm{cone}(\GC_{(g),1}\rightarrow \GCex_{(g)}) \rightarrow \icgg(1)[1]\rightarrow \Bar^c(\Graphsg(1))[1].
\]
The cocone of the last row may be identified via the adjoint action with $\Bar^c (\Graphsg(1))$. By Lemma \ref{lem:icgwg}, $\Bar^c(\Graphsg(1))$ is quasi-isomorphic to $\hwg_{g}^{fr}$ as dg Lie algebras.
\end{proof}

In particular, this concludes the proof of part (ii) of Theorem \ref{thm:GC comparison}.
\subsection{Comparison with higher genus Kashiwara-Vergne}
Let $m=1$. We briefly recall the definition of $\krv_{g,1} \subset \wgeis)$ from \cite{AKKN}. First, they define a Lie algebra 1-cocycle 
\[
\div \colon \Der (\wgeis ) \to |T\bar H(W_{g,1}))|
\]
where $|T\bar H(W_{g,1}))|$ is the space of cyclic words in $\bar H(W_{g,1})$ (equivalently, the $0$-th Hochschild homology of the tensor algebra $T \bar H(W_{g,1}) = U(\wgeis)$). It is defined by the formula
\[
\div(u) = | \partial_i u(c_i) |
\]
where $\partial_i \colon \wgeis \to T \bar H(W_{g,1})$ is defined by the property that $a \mapsto \ad_{\partial_i a}s \colon \wgeis \to \wgeis\langle s\rangle$ is the unique derivation that sends $c_i$ to the new symbol $s$. Using this, $\krv_{g,1}$ is given by
\[
\krv_{g,1} := \left\{ u \in \Der(\wgeis) \ \Bigg| \ u \left(\sum_{i,j=1}^{2g} g_{ij} [c_i, c_j] \right) = 0 , \, \div(u) = \left|f\left(\sum_{i,j=1}^{2g} g_{ij} [c_i, c_j]\right)\right| \text{ for some $f \in \Q[\![t]\!]$} \right\}.
\]
One also defines
\[
\krv_{g,1} \supset \widehat\krv_{g,1} := \left\{ u \in \Der(\wgeis) \ \Bigg| \ u \left(\sum_{i,j=1}^{2g} g_{ij} [c_i, c_j] \right) = 0 , \, \div(u) = 0 \right\}.
\]

\begin{prop}
\label{prop:gc to kv}
The image of the morphism \eqref{eq:GCg1 to Freelie} 
\[
H^0(\GC_{(g),1})\to \mathrm{Der}(\FreeLie(\bar H(W_{g,1})))
\]
is contained in $\widehat\krv_{g,1}$. That is, there is a natural Lie algebra morphism 
\[
H^0(\GC_{g,1})\to \widehat\krv_{g,1}.
\]
\end{prop}

\begin{proof}
The first defining equation roughly follows from the fact that the action of $\GC_{(g),1}$ on $\wgeis$ extends to $\wg_g$. Graphically, one can see this as follows. Let $\Gamma \in \GC_{(g),1}$ be a closed degree zero element and let $[\omega]$ denote the element in $\icgg(1)[1]$ without internal vertices and a single decoration $\omega$. Then
\[
\Gamma. d[\omega] = d (\Gamma.[\omega]).
\]
Note that $\Gamma.[\omega]$ does not contain any $\omega$-decorations and thus the equation is in $\icggeis(1)$ which shows the first defining equation of $\widehat{\krv}_{g,1}$.

To obtain the second equation we proceed as in \cite{severa}. We consider at the spectral sequence associated to the complete descending filtration given by the number of loops. On the associated graded complex the differential is given by the vertex splitting operation. In particular, we find that the $E_1^{0,0}$ term (i.e. no loops, degree zero) consists of trivalent trees modulo the ``IHX'' relation (obtained by splitting a single vertex of valence 4). These are naturally identified with derivations $\Der(\wgeis)$ satisfying the first equation in the definition of $\krv_{g,1}$. Similarly, considering $E_1^{1,0}$ which computes the lowest-degree cohomology of one-loop graphs (with respect to the vertex splitting differential) is given by trivalent one-loop graphs modulo ``IHX''. Thus, we may assume its representatives to be given by the graphs where the edges form one ``spinal'' loop and each vertex carries a single decoration.
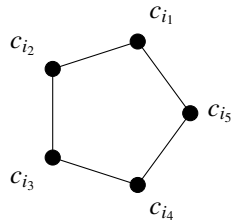
\[
\begin{tikzpicture}
\def\n{5}
\foreach \i in {1,...,\n}{
\node[int, label=\i*360/\n:{$c_{i_{\i}}$}] at (\i*360/\n:1) {};
\draw (\i*360/\n:1) edge ({(\i+1)*360/\n}:1);
}
\end{tikzpicture}
\]

These are thus identified up to symmetry (as in \cite{severa}) with cyclic words in $\bar H(W_{g,1})$, i.e.
\[
E_1^{1,0} \subset |T\bar H(W_{g,1}))|.
\]
Additionally, the differential $d_1:E_1^{0,0}\to E_1^{1,0}$ performs the glueing operation on the trivalent trees. This is equivalently described as follows. Mark two decorations in all possible ways. Consider their unique connecting path as the spine. Using the ``IHX'' relation, we may now rewrite the tree as a ``bird-on-a-wire'' tree (see the figure below, or \cite{NaefQin}).
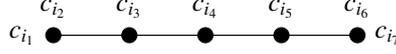
\[
\begin{tikzpicture}
\def\n{5}
\foreach \i in {3,...,\n}{
\node[int, label={$c_{i_{\i}}$}] at (\i-2,0) {};
\draw (0,0) edge (\n-1,0);
}
\node[int, label=180:{$c_{i_{1}}$}, label=90:{$c_{i_{2}}$}] at (0,0) {};
\node[int, label=90:{$c_{i_{6}}$}, label=360:{$c_{i_{7}}$}] at (\n-1,0) {};
\end{tikzpicture}
\]
Finally, we connect the two marked decorations ($c_{i_1}$ and $c_{i_7}$ in the figure above). (The only difference to the direct definition of $d_1$ is the order of connecting vs. using ``IHX''.)
Marking two decorations is equivalent to first marking one decoration and then another one. Marking all $c_i$-decorations and thinking of it as a rooted tree gives the image of
\[
u(\sum_{j=1}^{2g} g^{ij} c_j),
\]
where $u \in \Der(\wgeis)$ is the induced derivation. Marking another $c_i$ decoration and using ``IHX'' to rewrite it as a ``birds-on-a-wire'' tree is then exactly the graphical description of $\partial_i$. We have thus shown that the diagram
\[
\begin{tikzcd}
  E_1^{0,0} \ar[r, "d_1"] \ar[d,"\cong"] & E_1^{1,0} \ar[d] \\
  \Der(\wgeis) \ar[r, "\div"] & \left|T\bar H(W_{g,1}))\right|,
\end{tikzcd}
\]
commutes, and we have
\[
H^0(\GC_{(g),1})\rightarrow E_2^{0,0}\rightarrow \widehat{\krv}_{g,1}
\]
which shows the claim.
\end{proof}

Next, recall Enriquez' elliptic version of the Grothendieck-Teichm\"uller Lie algebra, $\grt_1^{ell}$ \cite{Enriquez}. Elements of $\grt_1^{ell}$ are of the form $(\phi, u)$ where $\phi \in \grt_1$ and $u\in \Der(\wgeis)$ satisfying a certain set of relations. He furthermore defines a $\grt_1^{ell}$-torsor $\underline{Ell}_1(k)$, the space of elliptic associators, whose elements are again pairs $(\Phi, U)$ where $\Phi \in \underline{M}(k)$ is a Drinfeld associator and $U$ an isomorphism $U \colon \wgeis \to k \bar{\pi}(W_{1,1})$ satisfying certain relations.
On the other hand, \cite{AKKN} define a $\krv_{1,1}$-torsor $\mathrm{SolKV_{1,1}}$, the space of solutions to the Kashiwara-Vergne problem in genus $(1,1)$, whose elements are isomorphisms $U \colon \wgeis \to k \bar{\pi}(W_{1,1})$ satisfying certain relations.

\begin{thm}
The assignment $(\phi, u) \mapsto u$ defines a Lie algebra map
\[
\grt^{ell}_1 \to \krv_{1,1}.
\]
Moreover, the assignment $(\Phi, U) \mapsto U$ defines a map
\[
\underline{Ell}_1(k) \to \mathrm{SolKV}_{1,1}.
\]
\end{thm}
\begin{proof}
That the projection $(\phi, \alpha_{\pm}) \mapsto (\alpha_+, \alpha_-)$ is a map of Lie algebras $\grt^{ell}_1 \to \Der(\wgeis)$ follows directly from the definition of the Lie algebra structure on $\grt^{ell}_1$ in \cite[Section 4.6]{Enriquez}. Moreover, by \cite[Proposition 4.15]{Enriquez} the projection $\grt^{ell}_1 \to \grt_1$ of Lie algebras has a section and \cite[Theorem 1.6]{AKKN} shows that the image of that section lands in $\krv_{1,1}$. Thus it remains to show that the kernel $\mathfrak r_{ell}$ of $\grt^{ell}_1 \to \grt_1$ is mapped to $\krv_{1,1}$. By \cite[Corollary 109]{Matteo}, $H^0(\GC_{(1)})$ is identified with the quotient of $\mathfrak{r}_{ell}$ by a central element. Since in turn $H^0(\GC_{(1),1})$ and $H^0(\GC_{(1)})$ differ by the corresponding element by Proposition \ref{prop:long exact splits}, we deduce that $\mathfrak r_{ell}\cong H^0(\GC_{(1),1})$. Finally, Proposition \ref{prop:gc to kv} shows that $H^0(\GC_{(1),1})$ is indeed mapped to $\krv_{1,1}$.

The moreover part now follows since both sides are torsors over the respective exponentiated groups corresponding to the Lie algebras from the first part. Thus it suffices to show that at least one element of $\underline{Ell}_1(k)$ is mapped into $\mathrm{SolKV}_{1,1}$ which is one of the the main results of \cite[Theorem 1.6]{AKKN}.
\end{proof}

\section{Degree bounds and proof of Theorem \ref{thm:main cohom GC vanishing}}

\subsection{Degree bounds on the "ordinary" graph complex}
We shall make use of degree bounds on the hairy graph cohomology that are shown in \cite[Theorem 1.6]{CGP2}.
The hairy graph complexes are defined similarly to the graph complexes considered in this paper, but they are not the same.
More precisely, analogously to the definitions in section \ref{sec:GCs} above, we may consider a complex $\HGC_{0,2}^{[h]}$ of (isomorphism classes of) connected at least trivalent graphs with $h$ hairs labeled $1,\dots,h$ and no tadpoles at vertices. The cohomological degree of such a graph $\Gamma\in\HGC_{0,2}^{[h]}$ is 
\[
(2\#\text{vertices}) - (\# \text{edges}).  
\]
Here we count the hairs as edges, but without a vertex at the end of the hair. For example, the graph 
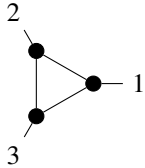
\[
\begin{tikzpicture} 
  \node[int] (v1) at (0:.5) {};
  \node[int] (v2) at (120:.5) {};
  \node[int] (v3) at (-120:.5) {};
  \node (w1) at (0:1.1) {$1$};
  \node (w2) at (120:1.1) {$2$};
  \node (w3) at (-120:1.1) {$3$};
\draw (v1) edge (v2) edge (v3) edge (w1)
(v2) edge (v3) edge (w2) 
(v3) edge (w3);
\end{tikzpicture}
\]
has three vertices, 6 edges and cohomological degree 0.
The differential $\delta=\delta_{split}$ is again given by splitting vertices.
The graph complex $\HGC_{0,2}^{[h]}$ appears at various places in the literature, in particular in the embedding calculus, and it also computes (up to degree shifts) the top weight cohomology of the moduli spaces of curves. The following degree bound is known.
\begin{prop}[Chan-Galatius-Payne, {\cite[Theorem 1.6]{CGP2}}]\label{prop:HGC bound}
The cohomology of $\HGC_{0,2}^{[h]}$ is concentrated in cohomological degrees $\geq -1$.
\end{prop}
\begin{proof}
To account for different degree conventions, let us briefly recall the argument of \cite{CGP2}. 
There the authors show that the part of the graph cohomology with $k$ internal (i.e., non-hair) edges, $h$ hairs and of loop order $g$ computes the weight 0 part of the compactly supported cohomology $H^k_c(\MM_{g,h})$ of the moduli space of curves.
But $\MM_{g,h}$ has real dimension $6g-6+2h$, and by a computation of Harer it has vcd $4g-4 +h +\delta_{0,g}-\delta_{0,h}$.
Hence we need to have $k\geq 2g-2 +h-\delta_{0,g}+\delta_{0,h}$.
Converting back to the degree $d$ in our graph complex we find 
\[
d = k - 2g - h +2 \geq -\delta_{0,g}+\delta_{0,h} \geq -1.
\]
\end{proof}
In fact, one sees from the proof that the bound is saturated for tree graphs only. 

\begin{rem}
The statement of Proposition \ref{prop:HGC bound} can alternatively be deduced from the proof of \cite[Lemmas 1.2, 1.4]{WillDegree}. More precisely, in \cite{WillDegree} only the part of $\HGC_{0,2}^{[h]}$ that is symmetric or antisymmetric under the interchange of labels $1,\dots,h$ is considered. However, the argument in that paper does not use the symmetrization, and goes through for the non-symmetrized complex as well. Furthermore, in \cite{WillDegree} graphs are allowed to have tadpoles, but this does not alter the result, as can be seen by the arguments of section \ref{sec:part i proof} above.
\end{rem} 

\subsection{Lower bound and proof of part (i) of Theorem \ref{thm:main cohom GC vanishing}}

In view of part (i) of Theorem \ref{thm:GC comparison}, that we have already proven and the weight 1 computation in section \ref{sec:weight 1 comparison}, it suffices to show part (i) of Theorem \ref{thm:main cohom GC vanishing} for $\GC_{(g),1}$.

Filter $\GC_{(g),1}$ by loop order, as in section \ref{sec:part i proof} above. Considering the associated spectral sequence, it then suffices to show that $H(\GC_{(g),1},\delta_{split})$ is concentrated in the degree range stated in part (i) of Theorem \ref{thm:main cohom GC vanishing}.

Consider first the case $m=1$, and pick some basis $c_1,\dots,c_{2g}$ of $H_1(W_{g,1})$. Then $(\GC_{(g),1},\delta_{split})$ splits into a direct product of subcomplexes according to the number of decorations $c_1,\dots,c_{2g}$ in graphs.
Concretely, fix some multi-index $\underline n=(n_1,\dots,n_{2g})$ and denote the respective subcomplex by $\GC_{(g),1}^{\underline n}\subset \GC_{(g),1}$. (It is spanned by graphs with $n_1$ decorations $c_1$, $n_2$ decorations $c_2$ etc.)
This subcomplex is then identified with a summand of the complex $\HGC_{0,2}^{[|n|]}$, concretely 
\begin{equation}\label{equ:GC HGC}
  \GC_{(g),1}^{\underline n} \cong (\HGC_{0,2}^{[|\underline n|]})_{S_{n_1}\times \cdots \times S_{n_{2g}} },
\end{equation}
with the symmetric groups acting by permuting hair labels. Invoking Proposition \ref{prop:HGC bound} this then shows that $H(\GC_{(g),1}^{\underline n},\delta_{split})$ is concentrated in non-negative degrees.
Hence the same holds for $H(\GC_{(g),1})$, finishing the proof for the case $m=1$.
Note that the complexes $\gr^W\GC_{(g),1}$ are isomorphic for any odd $m$, up to overall degree shifts. Concretely, changing $m$ by $+2$ reduces the degree of a given graph by $2W$.
Hence the vanishing statement of \ref{thm:main cohom GC vanishing} (i) follows for odd $m$.

For even $m$ one has to mind that the decorations by $H_m(W_{g,1})$ are now even objects. However, the same proof as above goes through, except that one has to antisymmetrize, instead of symmetrize over hair labels in \eqref{equ:GC HGC}.

In fact, in either case one sees from the proof of Proposition \ref{prop:HGC bound} that the degree bound is saturated (only) by trees, and in fact only trivalent such trees. 

\subsection{Proof of part (ii) of Theorem \ref{thm:main cohom GC vanishing}}
By part (ii) of Theorem \ref{thm:GC comparison} we know that for $g\geq 2$ we have that $H(\GCex_{(g)})$ is a summand (as a graded vector space) of $H(\GC_{(g),1})$, in a way preserving the weight grading.
Hence part (ii) of Theorem \ref{thm:main cohom GC vanishing} follows immediately from part (i).

\subsection{Graph complex with two-colored edges -- recollection from \cite{BM}}\label{sec:CGamma}
Let $\Gamma$ be a graph with vertex set $\{1,\dots,N\}$ and $k$ (ordered) edges.
Consider the complex 
\[
C^{\otimes k}    
\]
with $C$ a two-dimensional acyclic complex,
\[
C = (\Q[1] \to \Q).
\]
we interpret the natural basis elements of $C^{\otimes k}$ as assignments of ``colors"\footnote{Or rather dash-patterns, for compatibility with grayscale printing.} to edges of $\Gamma$, say the $\Q[1]$ corresponding to a solid edge, and $\Q$ to a dashed edge.
Hence $C^{\otimes k}$ has a natural basis that can be identified with colorings of the edges of $\Gamma$ such as the following.
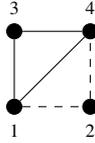
\[\begin{tikzpicture}
  \node[int,label=-90:{$\scriptstyle 1$}] (v1) at (0,0) {};
  \node[int,label=-90:{$\scriptstyle 2$}] (v2) at (1,0) {};
  \node[int,label=90:{$\scriptstyle 3$}] (v3) at (0,1) {};
  \node[int,label=90:{$\scriptstyle 4$}] (v4) at (1,1) {};
  \draw (v1) edge[dashed] (v2) edge (v3) edge (v4)
  (v2) edge[dashed] (v4)
  (v3) edge (v4);
  \end{tikzpicture}\]
  The differential acts by summing over edges, and replacing a solid edge by a dashed one.
It is clear that the complex $C^{\otimes k}$ is acyclic if $k\geq 0$, i.e., if $\Gamma$ has at least one edge.

Let $I_{disc}\subset C^{\otimes k}$ be the subcomplex spanned by all colored graphs for which the subgraph comprised of the vertices and the solid edges is not connected. We define 
\[
C_\Gamma = C^{\otimes k}/I
\] 
to be the quotient complex obtained by setting to zero all such solid-disconnected graphs.
We note that, obviously, $C_\Gamma$ is bounded and concentrated in cohomological degrees $-k,\dots, 1-N$.

\begin{lemma}[Lemma 8.8 of \cite{BM}]\label{lem:CGamma}
The cohomology of $C_\Gamma$ is concentrated in the top degree $1-N$.
\end{lemma}
The cohomology is hence given by linear combinations of solid trees, modulo the differential of solid one-loop graphs.

\begin{proof}
We proceed by induction on the number $N$ of vertices and the number $k$ of edges. The base cases (either $N=1$ or $k=0$) are trivial.
Suppose we know the statement for all graphs with $<N$ vertices or with $N$ vertices and $<k$ edges.

Pick some edge $(i,j)$. We may assume $i\neq j$ for clarity, for otherwise $C_\Gamma$ is obviously acyclic.
We have a splitting of graded vector spaces 
\[
  C_\Gamma = V_0 \oplus V_1,    
\]
where $V_0\subset C_\Gamma$ is spanned by graphs in which the edge $(i,j)$ is solid, and $V_1$ spanned by graphs where it is dashed.
Taking an appropriate spectral sequence the first page is
\[
H(V_0) \oplus H(V_1).
\]
But the complex $V_0$ is isomorphic to $C_{\Gamma/(i,j)}$, with $\Gamma/(i,j)$ the graph obtained by contracting the edge $(i,j)$.
Similarly $V_1$ is isomorphic to $C_{\Gamma\setminus(i,j)}$, with $\Gamma\setminus(i,j)$ obtained by removing edge $(i,j)$ from $\Gamma$. 
Invoking the induction hypothesis we are hence done.
\end{proof}

\subsection{Upper bound and proof of part (iii) of Theorem \ref{thm:main cohom GC vanishing}} \label{sec:main cohom GC vanishing proof 3}

We next want to show the vanishing statement Theorem \ref{thm:main cohom GC vanishing} (iii), thus finishing the proof of Theorem \ref{thm:main cohom GC vanishing}.
By using Theorem \ref{thm:GC comparison} it is sufficient to show the statement for $\GC_{(g),1}^\tp$, and for the other two complexes it then follows.
More precisely, By Theorem \ref{thm:GC comparison} (i) we know that $\gr^W\GC_{(g),1}^\tp=\gr^W\GC_{(g),1}$ for weights $W\geq 2$. But the weight $W=1$-part $\gr^W\GC_{(g),1}$ is concentrated in degree $1-m$.
Similarly, by Theorem \ref{thm:GC comparison} (ii) we have that $\gr^W H(\GCex_{(g)})$ is a summand of $\gr^W H(\GC_{(g),1})$ as long as $g\geq 2$, and by assumption we have $g\geq W+2\geq 3$.

Hence we need to show part (iii) of Theorem \ref{thm:main cohom GC vanishing} only for $\GC_{(g),1}^\tp$. In fact, we will equivalently show the dual vanishing statement for the graded dual complex $\G_{(g),1}^\tp$. Concretely we need to show that 
\[
\gr^W H^k(\G_{(g),1}^\tp)= H^k(\gr^W \G_{(g),1}^\tp) = 0  
\]
for $k<(m-1)W$ and $g\geq W+2$.
Let us filter $\gr^W\G_{(g),1}^\tp$ by the total number of vertices by setting $\mF^p \gr^W\G_{(g),1}^\tp$ to be the subspace spanned by graphs with $\leq p$ vertices. 
By finite dimensionality of the complex the associated spectral sequence converges to the cohomology.
The differential on the associated graded is just the part $d_{cut}$ of the full differential. Hence it suffices to show that 
\begin{equation}\label{equ:upper bound tbs}
  H^k(\gr^W \G_{(g),1}^\tp, d_{cut}) = 0\quad\quad \text{for $k< (m-1)W$ and $g\geq W+2$}.
\end{equation}

To go further, we would like to apply Lemma \ref{lem:complex inv}.
To this end we check:
\begin{lemma}\label{lem:upper bound D}
  Let $mD$ be the total degree of decorations appearing in a graph $\Gamma$ in $\gr^W \G_{(g),1}$, $\gr^W \G_{(g),1}^\tp$ or $\gr^W \G_{(g)}$. Then 
\[
  D \leq  W +2.
\]
The bound is saturated for tree graphs.
\end{lemma}
\begin{proof}
For a graph with $e$ edges, $v$ vertices and decorations of degree $mD$ the weight is by definition $W=2(e-v)+D$, and hence $D=W-2(e-v)$.
Since for a connected graph $e-v\geq -1$, with equality precisely for tree graphs, the result follows.
\end{proof}

We hence find that the $\OSp_g$-representation $\gr^W\G_{(g),1}^\tp$ is of order $\leq W+2$, in the notation of section \ref{sec:rep theory}, as is clear from the definition, see in particular \eqref{equ:fGdef}.

But by Lemma \ref{lem:complex inv} it then suffices to check that 
\[
  H^k((\gr^W \G_{(g),1}^\tp \otimes V_g^{\otimes M})^{\OSp_g}, d_{cut}) = 0
\]
for all $M=0,1,\dots, W+2$ and $k<(m-1)W$.
Suppose next that $g\geq W+2$.
Then we may apply classical invariant theory in the form of Theorem \ref{thm:inv theory}.
This then identifies  
\begin{equation}\label{equ:gr fc inv}
 ( (\gr^W \G_{(g),1}^\tp \otimes V_g^{\otimes M})^{\OSp_g}, d_{cut})
\end{equation}
with a complex of graphs with no decorations, but two types of edges, say solid and dashed, and $M$ external dashed legs. Here is an example graph for $M=3$ and $W=6$:
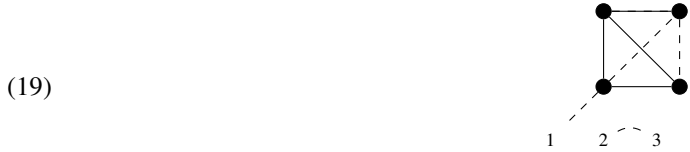
\begin{equation}\label{equ:dashed graph}
  \begin{tikzpicture}
    \node[int] (v1) at (0,0) {};
    \node[int] (v2) at (1,0) {};
    \node[int] (v3) at (0,1) {};
    \node[int] (v4) at (1,1) {};
    \node (w1) at (-.7,-.7) {$\scriptstyle 1$};
    \node (w2) at (0,-.7) {$\scriptstyle 2$};
    \node (w3) at (.7,-.7) {$\scriptstyle 3$};
    \draw (v1) edge (v2) edge (v3) edge[dashed] (v4) edge[dashed] (w1)
    (w2) edge[dashed, bend left] (w3)
    (v3) edge (v4) edge (v2)
    (v3) edge[dashed] (v4)
    (v4) edge[dashed] (v2);
  \end{tikzpicture}
\end{equation}
Each dashed edge represents a copy of the diagonal as in Theorem \ref{thm:inv theory}.
The differential $d_{cut}$ acts by replacing a solid edge by a dashed edge, as long as this leaves the solid subgraph connected.
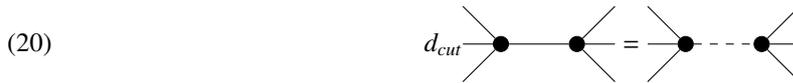
\begin{align}\label{equ:dcut dashed}
  d_{cut}
  \begin{tikzpicture}
    \node[int] (v) at (0,0) {};
    \node[int] (w) at (1,0) {};
    \draw (v) edge (w) 
          (v) edge +(-.5,-.5) edge +(-.5,.5) edge +(-.5,0)
          (w) edge +(.5,-.5) edge +(.5,.5) edge +(.5,0);
  \end{tikzpicture}
  &= 
  \begin{tikzpicture}
    \node[int] (v) at (0,0) {};
    \node[int] (w) at (1,0) {};
    \draw (v) edge[dashed] (w) 
          (v) edge +(-.5,-.5) edge +(-.5,.5) edge +(-.5,0)
          (w) edge +(.5,-.5) edge +(.5,.5) edge +(.5,0);
  \end{tikzpicture}
\end{align}

Note that the differential $d_{cut}$ does not alter the \emph{core} graph, by which we mean the graph obtained by forgetting the color (solid or dashed) of edges. Hence $( (\gr^W \G_{(g),1}^\tp \otimes V_g^{\otimes M})^{\OSp_g}, d_{cut})$ splits as a direct sum of subcomplexes according to core graphs $\Gamma$, each of which has the form 
\begin{equation}\label{equ:fc CGamma}
  (C_{\Gamma'} \otimes K(\Gamma))_{\Aut(\Gamma)},
\end{equation}
with $\Gamma'$ obtained from $\Gamma$ by removing the legs, $C_{\Gamma'}$ the two-colored (solid-dashed) complex from the previous subsection, $\Aut(\Gamma)$ the (finite) automorphism group of the core graph $\Gamma$ and $K(\Gamma)$ some one-dimensional representation of $\Aut(\Gamma)$ taking care of signs and degrees.
But by Lemma \ref{lem:CGamma} we know that $H(C_{\Gamma'})$ is represented by two-colored graphs whose solid subgraph is a spanning tree.
The same is hence true for any direct summand of $C_{\Gamma'}$, and in particular for \eqref{equ:fc CGamma}.
But a graph in $\G_{(g),1}^\tp$ of weight $W$ with $e$ edges has cohomological degree 
\[
k=mW-e-1.
\]
For tree graphs one has furthermore that $D=W+2$, and also $e\leq D-3=W-1$ from the condition that all vertices need to be at least trivalent. It follows that 
\[k\geq (m-1)W,\]
thus finishing the proof of part (iii) of Theorem \ref{thm:main cohom GC vanishing}.
\hfill \qed

\section{The Chevalley Eilenberg complex, and proof of Theorem \ref{thm:main CE all}}
\label{sec:CE}

\subsection{Definition, and graphical interpretation}
\label{sec:CE interpretation}
We define the Chevalley-Eilenberg complexes of the three dg Lie algebras considered in this paper as the cobar constructions of the respective graded dual Lie coalgebras
\begin{align*}
C_{CE}(\GC_{(g),1}) &= \Bar^c \G_{(g),1}
&
C_{CE}(\GC_{(g),1}^\tp) &= \Bar^c \G_{(g),1}^\tp
&
C_{CE}(\GCex_{(g)}) &= \Bar^c (\Gex_{(g)}).
\end{align*}
In particular, the Chevalley-Eilenberg complexes thus defined are dg commutative algebras, and graphically the commutative product is given by the disjoint union of graphs. 
They furthermore inherit the weight grading from the Lie algebras. 
The summands $\gr^W C_{CE}(\cdots)$ of fixed weight are furthermore finite dimensional dg vector spaces since the weight on the Lie algebras is positive.

In fact, by the very definition of the Lie bracket of our graphical dg Lie algebras, we can identify the Chevalley-Eilenberg complexes with the respective complexes of (possibly) non-connected graphs from section \ref{sec:GCs},
\begin{align*}
  C_{CE}(\GC_{(g),1}) &= \fG_{(g),1}
  &
  C_{CE}(\GC_{(g),1}^\tp) &= \fG_{(g),1}^\tp
  &
  C_{CE}(\GCex_{(g)}) &= \fGex_{(g)}.
\end{align*}
Here we should remark that the interpretation of $\fGex_{(g)}$ as a complex of graphs requires us to also provide a suitable graphical interpretation of the factors $\osp^{nil,c}_g\cong \Q^{2g}[-m]$.
In $\fGex_{(g)}$ we shall represent such a factor by a special (crossed) univalent vertex with a decoration in $H^m(W_g)$:
\[
\begin{tikzpicture}
\node[cross, label={$\alpha$}] (v) {};  
\end{tikzpicture}
\quad\quad \text{with $\alpha \in H^m(W_g)$}.
\] 
Here the special vertex formally carries degree $+1$ in the Chevalley-Eilenberg complex, so that the cohomological degree of a graph $\Gamma$ with $v$ (``normal") vertices, $c$ crossed vertices, $e$ edges and total decoration degree $mD$ is 
\[
  (2m-1) e - 2mv + mD +c.
\]
This formula applies to all three Chevalley-Eilenberg complexes, with $c=0$ in the cases $\fG_{(g),1}$ and $\fG_{(g),1}^\tp$.




We shall also give the combinatorial interpretation of the differential. The differential on the Chevalley-Eilenberg complex consists of two pieces, induced from the internal differential of the Lie coalgebra, and the Lie cobracket. These terms neatly unite in the Chevalley complex to give a total differential of the form $d=d_{c} + d_{cut} + d_{mul}''+ d_{\times}$, with the three pieces as follows.
The piece $d_c$ sums over all edges, contracting the edge. The piece $d_{cut}$ sums over all edges, replacing the edge by a pair of decorations. The piece $d_{mul}''$ is only present in the third of our three Chevalley-Eilenberg complexes (i.e., $C_{CE}(\GCex_{(g)})$) and acts by multiplying two decorations and contracting an adjacent edge as in section \ref{sec:GCs}.
The piece $d_{\times}$ is also only present in $C_{CE}(\GCex_{(g)})$, and acts by creating a crossed vertex.
There are two kinds of terms, the first coming from the Lie cobracket:
\begin{align*}
d_\times \, 
\begin{tikzpicture}
  \node[int,label={$\alpha$}] (v) at (0,0) {};
  \draw (v) edge +(.5,.5) edge +(.5,-.5) edge +(.5,0);
\end{tikzpicture}
\,
&=
\,
\begin{tikzpicture}
  \node[cross,label={$\alpha$}] (w) at (-0.7,0) {};
  \node[int] (v) at (0,0) {};
  \draw (v) edge +(.5,.5) edge +(.5,-.5) edge +(.5,0);
\end{tikzpicture}
&
d_\times \, 
\begin{tikzpicture}
  \node[int,label={$\omega$}] (v) at (0,0) {};
  \draw (v) edge +(.5,.5) edge +(.5,-.5) edge +(.5,0);
\end{tikzpicture}
\,
&=
\,
\sum_{i,j=1}^{2g}
g_{ij}
\begin{tikzpicture}
  \node[cross,label={$c_i$}] (w) at (-.7,0) {};
  \node[int,label={$c_j$}] (v) at (0,0) {};
  \draw (v) edge +(.5,.5) edge +(.5,-.5) edge +(.5,0);
\end{tikzpicture}
\end{align*} 
The second type of term comes from the piece of the differential $\G_{(g)}\to \osp_{g}^{nil, c}$, and this only acts on isolated vertices with exactly three decorations in $H^m(W_g)$,
\begin{align*}
  d_\times \, 
  \begin{tikzpicture}
    \node[int,label={$\alpha\beta \gamma$}] (v) at (0,0) {};
  \end{tikzpicture}
  \,
  &=
  \,
  \langle \alpha,\beta\rangle
  \begin{tikzpicture}
    \node[cross,label={$\gamma$}] (w) at (-0.7,0) {};
  \end{tikzpicture}
  +
  \langle \beta,\gamma\rangle
  \begin{tikzpicture}
    \node[cross,label={$\alpha$}] (w) at (-0.7,0) {};
  \end{tikzpicture}
  +
  \langle \gamma,\alpha\rangle
  \begin{tikzpicture}
    \node[cross,label={$\beta$}] (w) at (-0.7,0) {};
  \end{tikzpicture} 
  \quad \quad \quad \text{for }\alpha,\beta,\gamma\in H^m(W_g)\ .
  \end{align*}

Let us also introduce a second grading, the grading by $E$-number, on our complexes. On $C_{CE}(\GC_{(g),1})$ and $C_{CE}(\GC_{(g),1}^\tp)$ the $E$-number of some graph (in the graphical interpretation of the Chevalley-Eilenberg complex) is simply the number of edges.
For $C_{CE}(\GCex_{(g)})$ we define the $E$-number as the number of edges minus the number of special vertices "$\times$".
In any case, note that the cohomological degree is then expressed through the weight $W$ and $E$-number $E$ as 
\[
  mW -E.
\]
In particular note that the differential is homogeneous of degree $-1$ with respect to the $E$-number, and hence the cohomology inherits a grading by $E$-number.
In fact, one can think of the $E$ number as an $m$-independent shifted version of the homological degree, that will be convenient to use below.
Note that Theorem \ref{thm:main CE all} states that for $g$ large enough the cohomology of our Chevalley-Eilenberg complexes becomes concentrated in $E$-number 0.

\subsection{Proof of Theorem \ref{thm:main CE all} in the first two cases}\label{sec:main CE all proof 1}
Let us turn to the proof Theorem \ref{thm:main CE all} for the cases $C_{CE}(\GC_{(g),1})$ and $C_{CE}(\GC_{(g),1}^\tp)$. 
We will show the following refined version.
\begin{prop}\label{prop:main CE refined 1}
  We have that 
  \[
  \gr^W H^k_{CE}(\GC_{(g),1}) = \gr^W H^k_{CE}(\GC_{(g),1}^\tp) = 0
  \]
  if either of the following conditions is satisfied:
  \begin{enumerate}[(i)]
  \item $k>mW$.
  \item $k<mW$ and $g\geq 2k - (2m-3)W+1$.
  \end{enumerate}
  In particular if $g\geq 3W$ and $k\neq mW$, then one of the two conditions above is always satisfied.
\end{prop}

We first note that the $E$-number in $C_{CE}(\GC_{(g),1})$ and $C_{CE}(\GC_{(g),1}^\tp)$ is just the number $e$ of edges in the graph, and hence non-negative.
It follows that, trivially,
\[
  \gr^WH^k_{CE}(\GC_{(g),1}) = \gr^WH^k_{CE}(\GC_{(g),1}^\tp) = 0
\]
if $k>mW$, and hence the first statement of Proposition \ref{prop:main CE refined 1} follows as well.

For the second statement we will need invariant theory and the following Lemma.

\begin{lemma}\label{lem:max deco}
  The maximal number of decorations of a non-trivial graph $\Gamma \in C_{CE}(\GC_{(g),1}^\tp)$ (resp. $\Gamma \in C_{CE}(\GC_{(g),1})$) of weight $W$ and $E$-number (i.e., number of edges) $e$ is
  \[
    3W -2e.
  \]
  In particular the maximum number of decorations in any graph of weight $W$ is $3W$.
  The bound is saturated by a graph with no edges, and all vertices carrying three decorations.
\end{lemma}
\begin{proof}
Let $v$ be the number of vertices in a graph $\Gamma$ and $D$ the number of decorations.
We have that $W=2(e-v)+D$ and, by the trivalence condition, $3v\leq 2v+D$. Eliminating $v$ from this system of inequalities yields $D\leq 3W-2e$.
\end{proof}

We return to the proof of statement (ii) of Proposition \ref{prop:main CE refined 1}. In fact, we shall focus on the case of $\GC_{(g),1}^\tp$, with the statements for $\GC_{(g),1}$ being shown analogously. 

From Lemma \ref{lem:max deco} we see in particular that the representation $\gr^W C_{CE}(\GC_{(g),1}^\tp)$ of $\OSp_g$ is of order $\leq 3W$, and specifically the part of $E$-number $e$ is a representation of order $\leq 3W-2e$ in the sense of section \ref{sec:rep theory}.
Equivalently, using that the cohomological degree is $k=mW-e$, the degree $k$-part $\gr^W C_{CE}^k(\GC_{(g),1}^\tp)$ is of order $\leq 3W-2(mW-k)=2k-(2m-3)W$.

Let us define the complex of invariants 
\[
   J_{g,M,W} := \left(\gr^W C_{CE}(\GC_{(g),1}^\tp) \otimes V_g^{\otimes M} \right)^{\OSp_g}.
\]
To show the part (ii) of Proposition \ref{prop:main CE refined 1}, it suffices by Lemma \ref{lem:complex inv} to check that 
\[
H^k(J_{g,M,W}) =0
\]
for each $k<mW$ and $M\leq 3W-2e=2k-(2m-3)W$, provided $g\geq 3W-2e+1$.
Also note that one has natural maps of complexes
\[
J_{g+1,M,W} \to J_{g,M,W}.
\]
Furthermore, it follows from Theorem \ref{thm:inv theory} and Lemma \ref{lem:max deco} that these maps are isomorphisms on the degree $k$ part 
\[
J_{g+1,M,W}^k \to J_{g,M,W}^k
\]
if $2g\geq M+3W-2e=M+2k-(2m-3)W$.
We shall denote the stable limit by
\[
  J_{\infty,M,W} := J_{g,M,W}\quad \text{for any $g$ such that $2g\geq M+3W$}.
\]
Similarly to the proof in section \ref{sec:main cohom GC vanishing proof 3} the complex $J_{\infty,M,W}$ is identified with a complex of graphs with two kinds of edges, say "solid" and "dashed", but without decorations, such that each edge represents a copy of the diagonal element $\Delta_1$ to be inserted at its endpoints. There are precisely $M$ external dashed legs.
See \eqref{equ:dashed graph} for a picture. Note however that in contrast to section \ref{sec:main cohom GC vanishing proof 3} above the solid subgraph does now not need to be connected.

The differential is $d=d_c + d_{cut}$ with $d_c$ contracting an edge and $d_{cut}$ replacing it by a dashed one as in \eqref{equ:dcut dashed}.

\begin{lemma}
  The cohomology of $J_{\infty,M,W}$ is concentrated in degree $k=mW$. 
\end{lemma}
\begin{proof}
Using again a spectral sequence on the filtration by the number of vertices, it suffices to show  that $(J_{\infty,M,W}, d_{cut})$ is acyclic. 
However, note that in contrast to the proof in section \ref{sec:main cohom GC vanishing proof 3}
there is no connectivity requirement on the resulting graphs in the present case. 
We may hence introduce a homotopy 
\begin{gather*}
h : J_{g,M,W} \to J_{g,M,W} \\
\Gamma \mapsto h\Gamma = \sum_{e} \text{undash}_e(\Gamma),
\end{gather*}
where the sum is over all dashed edges $e$ \emph{between two vertices} (i.e., not external legs) and $\text{undash}_e(\Gamma)$ is the graph obtained by replacing the dashed edge $e$ by a normal one, which becomes the first in the ordering of edges to fix the sign.
One then verifies that 
\[
hd_{cut}\Gamma + d_{cut} h \Gamma= N(\Gamma)  \Gamma
\]
with $N(\Gamma)$ the number of edges (dashed or normal alike) between two vertices of $\Gamma$.
This means in particular that the cohomology is concentrated in $E$-number zero, or equivalently in cohomological degree $k=mW$.
\end{proof}

We return to the proof of Proposition \ref{prop:main CE refined 1}.
We know (from Theorem \ref{thm:inv theory}) that the map 
\[
  J_{\infty,M,W}^k \to J_{g,M,W}^k
\]
is an isomorphism for $2g\geq M + 2k - (2m-3)W$.
It follows that the induced map
\[
  H^k(J_{\infty,M,W}) \to H^k(J_{g,M,W})
\]
is an isomorphism for $2g\geq M + 2k+2 - (2m-3)W$.
In particular, if $g\geq 2k - (2m-3)W+1$ then
\[ 
H^k(J_{g,M,W}) = H^k(J_{\infty,M,W}) =0
\]
for each $k<mW$ and $M\leq 2k-(2m-3)W$.
This shows Proposition \ref{prop:main CE refined 1} for the case 
$C_{CE}(\GC_{(g),1}^\tp)$.

Note that the argument is agnostic of whether or not the graphs have tadpoles, and so the same proof goes through for the case $C_{CE}(\GC_{(g),1})$. \hfill \qed





\subsection{Proof of Theorem \ref{thm:main CE all} for $\GCex_{(g)}$}\label{sec:main CE all proof 2}
We shall proceed analogously to the previous subsection, with some additional complications. For simplicity, we shall also restrict to showing Theorem \ref{thm:main CE all} and not the refined genus bounds as in Proposition \ref{prop:main CE refined 1} in this case.

\begin{lemma}\label{lem:max deco 2}
  Let $\Gamma \in \gr^W C_{CE}(\GCex_{(g)})$ be a graph of weight $W$ with total degree of decorations $mD$. Then $D\leq 3W$. If $\Gamma$ has $E$-number $E$ then $D\leq 3W-2E$.
\end{lemma}
The bound $D\leq 3W$ is again saturated by a graph with no edges, and all vertices carrying three decorations.
\begin{proof}
Let $v$ be the number of (non-crossed) vertices in a graph $\Gamma$, $c$ the number of crossed vertices, $e$ the number of edges and $D$ the number of decorations.
To be concrete, a factor in $\osp_g^{(nil)}$ counts as zero non-crossed vertices, one crossed vertex, zero edges, and one decoration.
Then the weight of $\Gamma$ is $W=2(e-v)+D$.
By the trivalence condition we have $3v+c\leq 2v+D$. Eliminating $v$ from this system of inequalities yields $D\leq 3W-2e-2c=3W-2E-4c$, and this is $\leq 3W$ and $\leq 3W-2E$ since $e,c\geq 0$.
\end{proof}

As before, to show Theorem \ref{thm:main CE all} it suffices to show that 
\[
   K_{g,M,W} := \left(\gr^W C_{CE}(\GCex_{(g)}) \otimes V_g^{\otimes M} \right)^{\OSp_g}
\]
has cohomology concentrated in $E$-number $0$, for all $M\leq 3W$. Furthermore, for $2g\geq 3W+M$ we can again invoke Theorem \ref{thm:inv theory} to compute $K_{g,M,W}$ explicitly.
Again $K_{g,M,W}$ may be interpreted as a complex of graphs with two kinds of edges, solid and dashed, with the dashed edges corresponding to copies of the "reduced" diagonal element
\[
\Delta_1=\sum_{i,j=1}^{2g} g_{ij} c_i\otimes c_j \in H^m(W_g)\otimes H^m(W_g),
\]
or graphically 
\[
\begin{tikzpicture}
  \node[int] (v) at (0,0) {};
  \node[int] (w) at (1,0) {};
  \draw (v) edge[dashed] (w) 
        (v) edge +(-.5,-.5) edge +(-.5,.5) edge +(-.5,0)
        (w) edge +(.5,-.5) edge +(.5,.5) edge +(.5,0);
\end{tikzpicture}
=
\sum_{i,j} g_{ij} 
\begin{tikzpicture}
  \node[int, label=0:{$c_i$}] (v) at (0,0) {};
  \node[int, label=180:{$c_j$}] (w) at (1.4,0) {};
  \draw 
        (v) edge +(-.5,-.5) edge +(-.5,.5) edge +(-.5,0)
        (w) edge +(.5,-.5) edge +(.5,.5) edge +(.5,0);
\end{tikzpicture}
\]
In the following we always quietly assume that $2g\geq 3W+M$, so that we are in the "stable" situation.
On $K_{g,M,W}$ the differential is then $d=d_{c}+d_{cut}+d_{mul}''+d_\times$ with the pieces acting as follows
\begin{itemize}
\item $d_c$ sums over (solid) edges contracting them as before.
\item $d_{cut}$ sums over solid edges, replacing them by a diagonal $\Delta=\omega\otimes 1 +1\otimes \omega + \Delta_1$ (with $\omega$ being again the top cohomology class). If the cutting procedure produces a vertex $v$ of valence less than three (and hence necessarily of valence two), then the graph is set to zero.
\item $d_{mul}''$ removes a vertex with precisely 2 incident dashed half-edges and one solid half-edge, by contracting the solid edge, adding one $\omega$-decoration, and fusing the two dashed half-edges. If the dashed half-edges form a tadpole then they are instead removed and the graph is multiplied by $(-1)^m 2g$. 
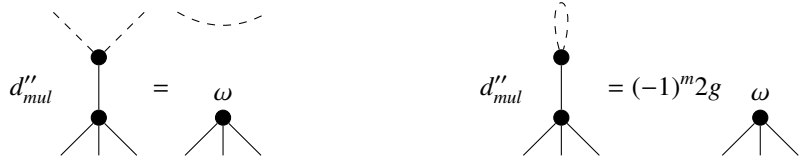
\begin{align*}
  d_{mul}''
\begin{tikzpicture}
  \node[int] (v) at (0,.4) {};
  \node[int] (w) at (0,-.4) {};
  \draw[dashed] (v) edge +(-.6,.6) edge +(.6,.6);
  \draw (v) edge (w) (w) edge +(-.5,-.5) edge +(0,-.5) edge +(.5,-.5); 
\end{tikzpicture}
&= 
  \begin{tikzpicture}
    \coordinate (v) at (-.6,1);
    \node[int,label=90:{$\omega$}] (w) at (0,-.4) {};
    \draw[dashed] (v) edge[bend right] +(1.2,0);
    \draw  (w) edge +(-.5,-.5) edge +(0,-.5) edge +(.5,-.5); 
  \end{tikzpicture}
  &
  d_{mul}''
  \begin{tikzpicture}
    \node[int] (v) at (0,.4) {};
    \node[int] (w) at (0,-.4) {};
    \draw[dashed] (v) edge[loop above, looseness=40] (v);
    \draw (v) edge (w) (w) edge +(-.5,-.5) edge +(0,-.5) edge +(.5,-.5); 
  \end{tikzpicture}
  &= 
  (-1)^m 2g
  \begin{tikzpicture}
    \node[int,label=90:{$\omega$}] (w) at (0,-.4) {};
    \draw  (w) edge +(-.5,-.5) edge +(0,-.5) edge +(.5,-.5); 
  \end{tikzpicture}
\end{align*}
\item The piece $d_\times$ comes from the $\osp_g^{nil,c}$-coaction on $\G_{(g)}$ and the part of the differential $\G_{(g)}\to \osp_{g}^{nil, c}$. The former piece replaces an $\omega$-decoration by a dashed edge towards a crossed vertex.
Furthermore, it replaces a dashed edge between internal vertices by a crossed vertex and one dashed edge, attached to either side. 
Pictorially
\begin{align}
  \label{equ:d times omega}
d_\times 
\begin{tikzpicture}
  \node[int, label=0:{$\omega$}] (v) at (0,0) {};
  \draw 
        (v) edge +(-.5,-.5) edge +(-.5,.5) edge +(-.5,0)
        ;
\end{tikzpicture}
&=
\begin{tikzpicture}
  \node[int] (v) at (0,0) {};
  \node[cross] (w) at (1,0) {};
  \draw (v) edge[dashed] (w)
        (v) edge +(-.5,-.5) edge +(-.5,.5) edge +(-.5,0)
        ;
\end{tikzpicture}
\\ 
\label{equ:d times ee}
d_\times 
\begin{tikzpicture}
  \node[int] (v) at (0,0) {};
  \node[int] (w) at (1,0) {};
  \draw (v) edge[dashed] (w) 
        (v) edge +(-.5,-.5) edge +(-.5,.5) edge +(-.5,0)
        (w) edge +(.5,-.5) edge +(.5,.5) edge +(.5,0);
\end{tikzpicture}
&=
\begin{tikzpicture}
  \node[int] (v) at (0,0) {};
  \node[int] (w) at (1.6,0) {};
  \node[cross] (x) at (.8,0) {};
  \draw (v) edge[dashed] (x)
        (v) edge +(-.5,-.5) edge +(-.5,.5) edge +(-.5,0)
        (w) edge +(.5,-.5) edge +(.5,.5) edge +(.5,0);
\end{tikzpicture}
\quad + \quad
\begin{tikzpicture}
  \node[int] (v) at (0,0) {};
  \node[int] (w) at (1.6,0) {};
  \node[cross] (x) at (.8,0) {};
  \draw (w) edge[dashed] (x)
        (v) edge +(-.5,-.5) edge +(-.5,.5) edge +(-.5,0)
        (w) edge +(.5,-.5) edge +(.5,.5) edge +(.5,0);
\end{tikzpicture}
\end{align}
Similarly, $d_\times$ acts on dashed edges that are not between two internal vertices in the same manner, and one drops the respective graphs if they are not in the complex. More explicitly, the complete list of cases is as follows, with an empty endpoint of the edge denoting an external leg.

\begin{align}\label{equ d times hair}
  d_\times 
  \begin{tikzpicture}
    \node[int] (v) at (0,0) {};
    \draw (v) edge[dashed] +(1,0)
          (v) edge +(-.5,-.5) edge +(-.5,.5) edge +(-.5,0)
          ;
  \end{tikzpicture}
  &=
  \begin{tikzpicture}
    \node[int] (v) at (0,0) {};
    \node[cross] (w) at (1,0) {};
    \draw (w) edge[dashed] +(1,0)
          (v) edge +(-.5,-.5) edge +(-.5,.5) edge +(-.5,0)
          ;
  \end{tikzpicture}
  \\
  d_\times 
  \begin{tikzpicture}
    \node[int] (v) at (0,0) {};
    \node[cross] (w) at (1,0) {};
    \draw (v) edge[dashed] (w) 
          (v) edge +(-.5,-.5) edge +(-.5,.5) edge +(-.5,0);
  \end{tikzpicture}
  &= 0
  \quad \quad 
  \text{(since }
  \begin{tikzpicture}
    \node[cross] (v) at (0,0) {};
    \node[cross] (w) at (.7,0) {};
    \draw (v) edge[dashed] (w);
  \end{tikzpicture} 
  =0 \text{ by symmetry)}
  \\
  d_\times 
  \begin{tikzpicture}
    \draw[dashed] (0,0)--(1,0);
  \end{tikzpicture}
  &=0.
  \end{align}
  All the above terms, coming from the Lie cobracket, introduce a new crossed vertex but do not alter the number of vertices otherwise.
  The terms stemming from the differential read:
  \begin{align*}
    d_\times
    \begin{tikzpicture}
      \node[int] (v) at (0,0) {};
      \draw[dashed] (v) edge +(30:.6) edge +(150:.6) edge +(-90:.6);
    \end{tikzpicture}
    &= 
    \pm
      \begin{tikzpicture}
        \coordinate (v) at (150:.6);
        \node[cross] (w) at (0,0) {};
        \draw[dashed] (v) edge (30:.6) (w) edge +(-90:.6); 
      \end{tikzpicture}
      +
      \begin{tikzpicture}
        \coordinate (v) at (30:.6);
        \node[cross] (w) at (0,0) {};
        \draw[dashed] (v) edge (-90:.6) (w) edge +(150:.6); 
      \end{tikzpicture}
      +
      \begin{tikzpicture}
        \coordinate (v) at (-90:.6);
        \node[cross] (w) at (0,0) {};
        \draw[dashed] (v) edge (150:.6) (w) edge +(30:.6); 
      \end{tikzpicture}
      &
      d_\times
      \begin{tikzpicture}
        \node[int] (w) at (0,0) {};
        \draw[dashed] (w) edge[loop above, looseness=40] (w) edge +(0,-.6);
      \end{tikzpicture}
      &= 
      (2+(-1)^m 2g)
      \begin{tikzpicture}
        \node[cross,label=90:{$\omega$}] (w) at (0,0) {};
        \draw[dashed]  (w) edge +(0,-.6); 
      \end{tikzpicture}
    \end{align*}
    Here open dashed edges can be connected to internal vertices, or be external legs.
\end{itemize}

To go further define the subspace $\tilde K_{g,M,W}\subset K_{g,M,W}$ to be spanned by graphs that contain one of the following features:
\begin{itemize}
\item A crossed vertex.
\item An $\omega$-decoration.
\item A connected component that is of the form 
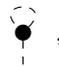
\[
\begin{tikzpicture}
  \node[int] (v) at (0,0) {};
  \draw (v) edge[loop, dashed] (v) edge[dashed] +(0,-.5);
\end{tikzpicture} ,
\]
that is, the connected component contains a single vertex with a dashed tadpole and a dashed external leg.
We will temporarily call such connected components \emph{islands} in the following.
\end{itemize}
We will temporarily call the features above the \emph{forbidden features}.

\begin{lemma}
The subspace $\tilde K_{g,M,W}\subset K_{g,M,W}$ is closed under the differential.
\end{lemma}
\begin{proof}
One just goes through the list of pieces of the differential and checks that no piece can remove one of the forbidden features above without introducing another. In other words the number of forbidden features is left the same or increased by the differential.
\end{proof}

\begin{lemma}
The projection to the quotient
\[
K_{g,M,W} \to K_{g,M,W} / \tilde K_{g,M,W}  
\]
is a quasi-isomorphism.
\end{lemma}
\begin{proof}
We introduce a descending filtration on $K_{g,M,W}$ by the number of forbidden features. Concretely, let $\mF^pK_{g,M,W}$ be spanned by graphs with $\geq p$ forbidden features. In particular $\mF^1K_{g,M,W}=\tilde K_{g,M,W}$.
It then suffices to check that the associated graded pieces $\gr^p K_{g,M,W}$ (with respect to the filtration by the number of forbidden features) are acyclic for $p\geq 1$. 
Considering the above description of the differential, the pieces contributing to the associated graded are $d_{cut}$, $d_c$, and the pieces \eqref{equ:d times omega} and \eqref{equ d times hair} of $d_\times$. Let us denote these latter pieces of $d_\times$ by $d_\times'$. We hence have to show that the complex $(\tilde K_{g,M,W},d_c+d_{cut}+d_\times')$ is acyclic.
By taking a further associated graded with respect to the increasing filtration by the number of solid edges it in fact suffices that $(\tilde K_{g,M,W},d_\times')$ is acyclic.


But this can be achieved by introducing the homotopy 
\begin{gather*}
h: \tilde K_{g,M,W} \to \tilde K_{g,M,W} \\
h(\Gamma) = \sum_{x} G_x(\Gamma)
\end{gather*}
with the sum running over all crossed vertices $x$, and $G_x(\Gamma)$ being defined by replacing the vertex $x$ according to the following rule:
\begin{itemize}
\item If $x$ is connected to an internal vertex, remove $x$ and place a decoration $\omega$ at the internal vertex.
\item If $x$ is connected to an external leg, remove $x$, glue in an island instead, and multiply the graph by $\frac 1 {2+(-1)^m 2g}$. Note that since $W\geq 1$ and $g\geq 3W$ we automatically have $2+(-1)^m 2g\neq 0$.
\end{itemize}

One then checks by explicit computation that 
\[
(d_\times'h+hd_\times')(\Gamma) = N(\Gamma) \Gamma,  
\]
with $N(\Gamma)$ being the number of forbidden features.
Since by construction $N(\Gamma)>0$ for any graph $\Gamma\in\tilde K^L$ we are done. 
\end{proof}

Note that the remaining pieces of the differential on $K_{g,M,W}/\tilde K_{g,M,W}$ are $d_c+d_{cut}$, since the $d_x$-terms produce a crossed vertex, and $d_{mul}''$ creates an $\omega$-decoration.
We may now proceed similarly to the previous subsection to compute the cohomology of $K_{g,M,W}/\tilde K_{g,M,W}$.
We define a filtration on $K_{g,M,W}/\tilde K_{g,M,W}$ by defining $\mF^p K_{g,M,W}/\tilde K_{g,M,W}$ to be spanned by graphs with $\leq p$ edges (solid or dashed). The associated graded is identified with $(K_{g,M,W}/\tilde K_{g,M,W},d_{cut})$. As in the previous subsection this differential has a homotopy and we conclude that the inclusion $\mF^0 K_{g,M,W}/\tilde K_{g,M,W}\to K_{g,M,W}/\tilde K_{g,M,W}$ is a quasi-isomorphism.
Since $\mF^0 K_{g,M,W}/\tilde K_{g,M,W}$ is concentrated in $E$-number zero (since it is spanned by graphs with no solid edges and no crossed vertices), we find that the spectral sequence abuts by degree reasons.
Overall we have shown that $H(K_{g,M,W})$ is concentrated in $E$-number zero, thus showing the remaining piece of Theorem \ref{thm:main CE all}.\hfill\qed

\subsection{Formality and presentation}

Let $C$ temporarily denote any of the three Chevalley-Eilenberg complexes above. Denote by $\tru^{[m]} (C)$ the truncation as defined in \eqref{equ:tru sq alpha}. Note that $\tru^{[m]} (C)$ consists of the elements of $C$ of positive $E$-number, and the closed elements of $E$-number zero.
Using the notation $H^{[m]}(C)$ (see \eqref{equ:H sq alpha}) for the part of the cohomology of $C$ of $E$-number zero we have a zig-zag of dg commutative algebras 
\begin{equation}\label{equ:Czigzag}
 C \leftarrow \tru^{[m]} (C) \to H^{[m]}(C)\subset H(C).
\end{equation}
Theorem \ref{thm:main CE all} then implies that this zigzag of dg commutative algebra morphisms induces an isomorphism on the pieces of the cohomology of weight $W\leq \frac g3$. In particular $C$ is formal up to weight $g/3$.
Furthermore note that since the $E$-number for the cases of $C=C_{CE}(\GC_{(g),1})$ and $C=C_{CE}(\GC_{(g),1}^\tp)$ is nonnegative, we have that $\tru^{[m]} (C)=C$ in these cases. Together we find the following Corollary of Theorem \ref{thm:main CE all}.

\begin{cor}
The morphisms of dg commutative algebras 
\begin{align*}
  C_{CE}(\GC_{(g),1}) \to H^{[m]}(C_{CE}(\GC_{(g),1})) =: \tilde A_{(g),1} \\
  C_{CE}(\GC_{(g),1}^\tp) \to H^{[m]}(C_{CE}(\GC_{(g),1}^\tp)) =: \tilde A_{(g),1}^\tp\\
  C_{CE}(\GCex_{(g)}) \leftarrow 
  \tru^{[m]} C_{CE}(\GCex_{(g)})
  \to H^{[m]}(C_{CE}(\GCex_{(g)}))=: \tilde A_{(g)}
\end{align*}
induce isomorphisms on the weight $W$-pieces of the cohomology for $g\geq 3W$.
\end{cor}

In particular, the $E=0$-parts $\tilde A_{(g),1}=H^{[m]}(C_{CE}(\GC_{(g),1}))$ and $\tilde A_{(g),1}^\tp=H^{[m]}(C_{CE}(\GC_{(g),1}^\tp))$ of the cohomology are given by graphs without edges, modulo the differential of graphs with exactly one edge.
Note also that any graph without edges is just a union of single vertices carrying at least three decorations. A graph with a single edge in addition either has a pair of vertices connected by that edge, or the edge forms a tadpole.
In particular, one has the following presentations
\begin{itemize}
  \item $\tilde A_{(g),1}^\tp$ is generated as a graded commutative algebra by the graded vector space $S^{\geq 3} (H^m(W_{g,1})[-m])$. We denote the generator corresponding to a monomial $a_1\cdots a_k\in S^{\geq 3} (H^m(W_{g,1})[-m])$ by $c(a_1\cdots a_k)$.
  Then the relations read, for $(e_i)$ a basis of $H^m(W_{g,1})$ and $(e_i^*)$ the Poincar\'e dual basis,
  \begin{align}\label{equ:Agrel1}
    c(a_1\cdots a_k) &= 
    \sum_{i=1}^{2g} c(a_1\cdots a_l e_i)
    c(e_i^* a_{l+1}\cdots a_k)
    &\text{for $l=2,\dots k-2$} \\
    \label{equ:Agrel2}
    \sum_{i=1}^{2g}  c(a_1\cdots a_ke_ie_i^*) &= 0.
  \end{align}
  \item $\tilde A_{(g),1}^\tp$ is also generated as a graded commutative algebra by the graded vector space $S^{\geq 3} (H^m(W_{g,1})[-m])$, but only with the relation \eqref{equ:Agrel1}, but not \eqref{equ:Agrel2}.
\end{itemize}

\subsection{Stabilization and degrees}
We note that one has the following lower degree bounds on the complexes above.

\begin{cor}
  \begin{enumerate}[(i)]
\item The maximum $E$-number of a graph $\Gamma$ of weight $W$ in $C_{CE}(\GC_{(g),1})$ or $C_{CE}(\GC_{(g),1}^\tp)$ or $C_{CE}(\GCex_{(g)})$ is $\frac 32 W$.
\item The minimal cohomological degree of a graph $\Gamma$ of weight $W$ in $C_{CE}(\GC_{(g),1})$ or $C_{CE}(\GC_{(g),1}^\tp)$ or $C_{CE}(\GCex_{(g)}$ is $(m-\frac 32)W$.
\end{enumerate}
\end{cor}
\begin{proof}
Let $D$ be the number of decorations and $E$ the $E$-number of $\Gamma$. Then by Lemmas \ref{lem:max deco} and \ref{lem:max deco 2} we have the inequality $D\leq 3W-2E$. Hence $2E\leq 3W-D\leq 3W$ since $D\geq 0$. This shows the first statement. The second follows from the formula for the cohomological degree $k$, $k=mW-E$.
\end{proof}

Using this result we see that for $m\geq 2$ and each fixed $k$ only finitely many weights $W$ contribute to the degree $k$ piece of our Chevalley-Eilenberg complexes.
Hence the degree $k$ cohomology also stabilizes (in the sense discussed) for $g\to \infty$. Here we just state for reference the corresponding stabilization result in degree that is an immediate reformulation of Proposition \ref{prop:main CE refined 1}.

\begin{cor}
  Suppose that $m\geq 2$. 
Then the maps 
\begin{align*}
  H_{CE}^k(\GC_{(g),1}) &\to \tilde A_{(g),1}^k \\
  H_{CE}^k(\GC_{(g),1}^\tp) &\to (\tilde A_{(g),1}^\tp )^k 
\end{align*}
are isomorphisms as soon as $g\geq 2k$.
\end{cor}
\begin{proof}
One just notes that $(2m-3)W\geq W\geq 1$ in the statement of Proposition \ref{prop:main CE refined 1}.
\end{proof}



\subsection{Invariant part}
Setting $M=0$ in the proofs of sections \ref{sec:main CE all proof 1} and \ref{sec:main CE all proof 2} we see that we have in particular computed the cohomology of the $\OSp_g$-invariant piece of the Chevalley-Eilenberg complex.

\begin{prop}
  \begin{enumerate}
\item For $2g\geq 2k-(2m-3W) +2$ or $k>mW$ we have that $\gr^W H^k_{CE}( \GC_{(g),1}^\tp)^{\OSp_g} =0$.
\item 
Let $\kappa_j$ represent the graph 
\[
\kappa_j  
=
\begin{tikzpicture}
\node[int, label={$\Delta_1^{j+1}$}] (v) at (0,0) {};
\end{tikzpicture}
\]
Then the map
$$\gr^W\Q[\kappa_1,\kappa_2,\dots]\to \gr^W H_{CE}(\GC_{(g),1})^{\OSp_g}$$ is an isomorphism in degree $k$ if $2g\geq 2k-(2m-3W) +2$ or $k>mW$. 
\item   
Similarly, the map 
$$
\gr^W\Q[\kappa_1,\kappa_2,\dots]\to \gr^W H_{CE}(\GCex_{(g)})^{\OSp_g}$$
is an isomorphism for $2g\geq 3W$.   
  \end{enumerate}
\end{prop}
\begin{proof}
One follows again the proofs of sections \ref{sec:main CE all proof 1} and \ref{sec:main CE all proof 2}, but with $M$ in these proofs set to zero.
The condition on $g$ (for the cohomology to agree with the stable one) in section \ref{sec:main CE all proof 1} then becomes $2g\geq M+2k-(2m-3W) +2=2k-(2m-3W) +2$. For the less refined argument of section \ref{sec:main CE all proof 2} it becomes $2g\geq 3W$.

Hence it suffices to consider the stable cohomology for $M=0$. Revisiting the proofs above this is computed to be 0 for $\GC_{(g),1}^\tp$, and (freely)
 spanned by classes without any edges (i.e., unions of the $\kappa_j$'s) for the other two cases.
\end{proof}





\section{Koszulness and quadratic presentations of the cohomology}\label{sec:Koszul}

We have shown so far that for large enough $g$ the cohomology of the weight $W$ part of our dg Lie algebras $\GC_{(g),1}$, $\GC_{(g),1}^\tp$ and $\GCex_{(g)}$ becomes concentrated in the single degree $(1-m)W$.
We have also shown that the cohomology of the Chevalley-Eilenberg complex becomes concentrated in degree $mW$.
In this section we shall derive quadratic presentations of the cohomology in those degrees, and show that both cohomologies (i.e., the cohomology of the Lie algebra, and that of its Chevalley-Eilenberg complex) are Koszul dual (in a range, for $g$ large).

\subsection{A version of Proposition \ref{prop:cohom conc follows koszul}}
We first formulate a slightly technical version of Proposition \ref{prop:cohom conc follows koszul} that is appropriate for our setting.

\begin{prop}\label{prop:cohom conc follows koszul 2}
  Let $W_0\geq 2$ and $\alpha$ be integers.
  Suppose that $\fg$ is a dg Lie algebra with an additional positive (``weight") grading, satisfying the following properties.
  \begin{itemize}
    \item Each graded piece $\gr^W\fg$ is finite dimensional.
    \item $\gr^W H(\fg)$ is concentrated in cohomological degree $\alpha W$ for any $W\leq W_0$.
    \item The cohomology of the weight $W$ piece of the cobar construction $\gr^W \Bar^c \fg^c$ is concentrated in cohomological degree $(1-\alpha)W$ for any $W\leq W_0$.
  \end{itemize}
  Then the following hold:
  \begin{enumerate}[(i)]
    \item $\fg$ is formal up to weight $W_0$ in the sense that the zigzag of morphisms of dg Lie algebras
    \[
    \fg \leftarrow \tru^{[\alpha]}\fg \to H^{[\alpha]}(\fg)   
    \] 
    (see \eqref{equ:tru sq alpha}, \eqref{equ:H sq alpha} for the notation) induces a zigzag of quasi-isomorphisms of complexes on the part of weight $W$ for any $W\leq W_0$.
    \item Similarly, the dg commutative algebra $\Bar^c \fg^c$ is formal up to weight $W_0$ in the sense that the zigzag of morphisms of dg commutative algebras 
    \[
      \Bar^c \fg^c \leftarrow \tru^{[1-\alpha]}
      \Bar^c \fg^c \to H^{[1-\alpha]}(\Bar^c \fg^c)
    \]
    induces quasi-isomorphisms on the part of weight $W\leq W_0$.

    \item Let $V:=\gr^1 H(\fg)\cong \gr^1 H(\Bar \fg)[-1]$, $R:= \gr^2H(\Bar \fg)$ and $S:=\gr^2H(\fg^c)$. Then the maps 
    \begin{align}\label{equ:kos2 inj}
      R &\to \wedge^2V & S\to S^2(V[-1]^*)  
    \end{align}
    given by the (reduced) coproduct and the cobracket are injections. We denote the images by $R$ and $S$ as well, abusing notation.
    Furthermore, $S$ is (up to degree shift) the annihilator of $R$ if we identify $(\wedge^2V)^*\cong S^2(V[-1]^*)[2]$.
    \item The graded Lie algebra $\ft$ and graded commutative algebra $A$ defined via the quadratic presentations
    \begin{align*}
      \ft &= \FreeLie(V)/R & 
      A= S(V^*[-1]) / S.
    \end{align*}
    are Koszul dual to each other.
    \item The maps of graded Lie, respectively graded commutative algebras
    \begin{align}\label{equ:xxx3}
      \ft &\to H(\fg) & 
      A&\to H(\Bar^c \fg^c)
    \end{align}
    defined by the obvious inclusion of generators are well-defined, respect the weight grading and induce isomorphisms on the parts of weight $W$ for each $W\leq W_0$.
    \item The pair $\ft$ and $A$ are Koszul up to weight $W_0$ in the sense that the cohomology of $\gr^W\Bar \ft$ (respectively $\gr^W\Bar A$) is concentrated in cohomological degree $(\alpha-1) W$ (respectively $-\alpha W$) for $W\leq W_0$.
  \end{enumerate}
\end{prop}
\begin{proof}
Items (i) and (ii) are immediate by the assumptions on the cohomology in weights $W\leq W_0$.

By (i) and (ii) we then know that the following zigzags of dg Lie or dg commutative algebras 
\begin{gather}\label{equ:xxx2a}
\fg \xleftarrow{\sim} \Bar^c\Bar \fg
\to \Bar^c (\tru^{[1-\alpha]}\Bar^c \fg^c)^c
\leftarrow \Bar^c (H^{[1-\alpha]}(\Bar^c \fg^c))^c \\
\Bar^c \fg^c \to
\Bar^c((\tru^{[\alpha]} \fg)^c)
\leftarrow 
\Bar^c(H^{[\alpha]} (\fg)^c)
\label{equ:xxx2b}
\end{gather}
are quasi-isomorphisms up to weight $W_0$.
By assumption the cohomology of the weight $W$ part of the left-hand sides is concentrated in degree $\alpha W$ 
(resp. $(1-\alpha)W$) for $W\leq W_0$.

Proof of (iii): Since $W_0\geq 2$ by assumption this means in particular that $\gr^2\Bar^c(H^{[\alpha]} (\fg)^c)$ is concentrated in degree $2-2\alpha$. 
But the complex $\gr^2\Bar^c(H^{[\alpha]} (\fg)^c)$ has the form 
\begin{equation}\label{equ:xxx1}
  \gr^2 H(\fg)^*[-1]
\to 
  S^2(V^*[-1]),
\end{equation}
with the differential given by the cobracket, the left-hand term living in degree $1-2\alpha$, and the right-hand term living in degree $2-2\alpha$. Since the cohomology is concentrated in degree $2-2\alpha$ the differential must be injective.
Similarly, one checks that the coproduct map of \eqref{equ:kos2 inj} is injective.
Furthermore, the cohomology of \eqref{equ:xxx1} is $R^*$, and is identified with the dual of the cokernel of the differential, that is the annihilator of the image $S$. Hence the final statement of (iii) of the proposition follows.

(iv) is immediate from (iii).

(v) It is clear that the maps preserve the weight grading.
To check that the maps are well-defined we just have to check that the relations $R$ (resp. $S$) are sent to zero, and this can be checked in the part of weight $2$. But the weight 2 part of $H(\Bar^c\fg^c )$ is the cohomology of \eqref{equ:xxx1}, and hence the image of $S$ is zero essentially by definition.
Similarly, the weight 2 part of $H(\fg)$ is the cohomology of $\gr^2\Bar^c H(\Bar\fg)$ and again $R$ is identified with the image of the differential, and hence zero in cohomology.

Next, consider the cobar construction $\Bar^c (H^{[1-\alpha]}(\Bar^c \fg^c)^c)$. In weight $W$ its piece of top cohomological degree is identified with $\gr^W\FreeLie(V)$, in degree $\alpha W$. The piece of top-minus-1-degree has precisely one factor of $R$. The top degree cohomology is hence identified with $\gr^W \ft$, so that 
\begin{equation*}
  \ft=H^{[\alpha]}(\Bar^c H^{[1-\alpha]}(\Bar^c \fg^c)^c).
\end{equation*}
Since we know that the zigzag \eqref{equ:xxx2a} of dg Lie algebra morphisms induces a zigzag of isomorphisms in weight $\leq W_0$-cohomology, we obtain that $\gr^W\ft\cong \gr^W H(\fg)$ for $W\leq W_0$.
This isomorphism is the identity $V\cong V$ in weight $1$ and the Lie brackets are respected, at least up to weight $W_0$. Hence it agrees with the fist map of \eqref{equ:xxx3}, which hence also is an isomorphism in weights $\leq W_0$.
One argues similarly for the second map of \eqref{equ:xxx3}, which agrees with the top degree cohomology map induced by \eqref{equ:xxx2b} in weights $\leq W_0$, and is hence an isomorphism in these weights.

For (vi) note that by (v) $\Bar \fg$ and $\Bar\ft$ are quasi-isomorphic in weights $\leq W_0$.
But the statement follows directly from the assumption on the cohomology of $\Bar \fg$.
Similarly, $\Bar A$ is quasi-isomorphic to $\Bar H(\Bar^c \fg)=\Bar H^{[1-\alpha]}(\Bar^c \fg)$ in weights $\leq W_0$  by (v).
Hence, since \eqref{equ:xxx2a} consists of quasi-isomorphisms in weights $\leq W_0$, the remaining statement of (vi) follows by the assumption on $H(\fg)$.
\end{proof}

\subsection{Main results}\label{sec:Koszul results}
Suppose that $g\geq 6$. We can apply the general construction of the previous subsection to the case of $\fg$ being either of $\GC_{(g),1}$, $\GC_{(g),1}^\tp$ or $\GCex_{(g)}$.
In each case it is shown in Theorem \ref{thm:main cohom GC vanishing} that the cohomology of $\gr^W\fg$ is concentrated in cohomological degree $\alpha W$ with $\alpha=(1-m)$, in the range $W\leq W_0=g/3\geq 2$ (since $3W\geq W+2$).
Furthermore, we know by Theorem \ref{thm:main CE all} that in the same $W$-range the cohomology of $\gr^W\Bar^c \fg^c$ is concentrated in degree $(1-\alpha)W=mW$.
Hence Proposition \ref{prop:cohom conc follows koszul 2} is applicable and we find that $H(\fg)$ and $H_{CE}(\fg)$ form a Koszul pair, in the given weight range. 
The spaces of generators are then
\begin{align*}
  V_{(g),1}^\tp 
  &= 
  \gr^1 H(\GC_{(g),1}^\tp)
   =\gr^1 H_{CE}(\GC_{(g),1}^\tp)^*[1]  \\
  V_{(g),1} 
  &= 
  \gr^1 H(\GC_{(g),1}) 
  = \gr^1 H_{CE}(\GC_{(g),1})^*[1]
  \\
  V_{(g)} 
  &= 
  \gr^1 H(\GCex_{(g)})
  =
  \gr^1 H_{CE}(\GCex_{(g)})^*[1].
\end{align*}
The spaces of relations can be identified with 
\begin{align*}
  R_{(g),1}^\tp &= \gr^2 H(A_{(g),1}^\tp)^* 
  \\
  R_{(g),1} &= \gr^2 H(A_{(g),1})^*
  \\
  R_{(g)} &= \gr^2 H(A_{(g)})^* 
\end{align*}

We hence find the following result, extending Theorem \ref{thm:main t}.
\begin{thm}\label{thm:presentations}
The morphisms of graded Lie algebras
\begin{align*}
\ft_{(g),1}&:=\FreeLie(V_{(g),1}) / \langle R_{(g),1} \rangle \to H(\GC_{(g),1})  \\
\ft_{(g),1}^\tp&:=
\FreeLie(V_{(g),1}^\tp) / \langle R_{(g),1}^\tp \rangle \to H(\GC_{(g),1}^\tp) \\
\ft_{(g)}&:=
\FreeLie(V_{(g)}) / \langle R_{(g)} \rangle
\to H(\GC_{(g)}^{\ex}) 
\end{align*}
and the morphisms of graded commutative algebras 
\begin{align*}
  A_{(g),1}&:=\FreeCom(V_{(g),1}^*[1]) / \langle R_{(g),1}^\perp \rangle \to H_{CE}(\GC_{(g),1})  \\
  A_{(g),1}^\tp&:=
  \FreeCom((V_{(g),1}^\tp)^*[1]) / \langle (R_{(g),1}^\tp)^\perp \rangle \to H_{CE}(\GC_{(g),1}^\tp) \\
  A_{(g)}&:=
  \FreeCom(V_{(g)}^*[1]) / \langle R_{(g)}^\perp \rangle
  \to H_{CE}(\GC_{(g)}^{\ex}) 
  \end{align*}
defined on generators by the obvious inclusion induce isomorphisms on the parts of weight $W$ as long as $g\geq 3W\geq 6$. Furthermore, the pairs $(\ft_{(g),1}, A_{(g),1})$ and $(\ft_{(g),1}^\tp, A_{(g),1}^\tp)$ and $(\ft_{(g)}, A_{(g)})$ are Koszul in the same range of weights, in the sense of Proposition \ref{prop:cohom conc follows koszul 2}.
\end{thm}
\hfill\qed

\subsection{Computation of generators and relations}
It remains to compute explicitly the spaces of generators and relations appearing in Theorem \ref{thm:presentations}. Concretely, in this section we will compute their decomposition into irreducible representations of $\OSp_g$. We assume again that $g\geq 6$.
There will be 6 cases to be considered, namely each of the three dg Lie algebras we study, either for odd or even $m$.
However, to find the spaces of generators we may re-use the computation of the weight-1-cohomology of section \ref{sec:weight 1 comparison}. 

As before, let $\lambda_1,\dots,\lambda_g$ be the usual system of fundamental weights.
We shall use the notation $V(\lambda)$ to denote an irreducible representation of $\OSp_g$ of highest weight $\lambda$.
Note that when $m$ is even $\OSp(g)=O(g,g)$ is not connected, and the irreducible representation is not uniquely determined by the highest weight.
However, one may construct irreducible representations of $O(g,g)$ associated to Young diagrams by using Weyl's construction, see e.g. \cite[19.5]{FultonHarris} or \cite[11.6.5]{Procesi}. 
We shall eventually only need the irreducible representations of $O(g,g)$ of the form $V(j \lambda_1 + k \lambda_2)$, which we define as the irreducible representation associated to a Young diagram with $k$ columns of length 2 and $j$ columns of length 1. 

The decompositions of several tensor products of representations stated below without proof have been computed using the SageMath computer algebra software.

{\bf Case $\GC_{(g),1}^\tp$: }
The complex $\gr^1 \GC_{(g),1}^\tp$ has the form 
\[
S^3(V_g[m])[-2m-1] \to V_g[-2m-2]
\]
with $V_g\cong H_m(W_{g,1})$ and the (surjective) differential obtained by contraction with the canonical bilinear form.
Graphically this corresponds to the differential 
\[
\begin{tikzpicture}
  \node[int,label={$\alpha\beta \gamma$}] (v) at (0,0) {};
\end{tikzpicture}  
\, \to\,  
\begin{tikzpicture}
  \node[int,label=-90:{$\delta$}] (v) at (0,0) {};
  \draw (v) edge [loop] (v);
\end{tikzpicture}  
\quad \quad \text{with $\delta=\langle\alpha,\beta\rangle \gamma
\pm \langle\alpha,\gamma\rangle 
\pm \beta \langle\beta,\gamma\rangle \alpha$}
\]
The cohomology is easily computed, see section \ref{sec:weight 1 comparison}, and in the notation of that section we find that the space of generators of our graded Lie algebra $\ft_{(g),1}^\tp$ is, as $\OSp_g$-representation,
\[
  V_{(g),1}^\tp := \gr^1 H(\GC_{(g),1}^\tp)
  =
  \begin{cases}
 V(\lambda_3)[m-1] & \text{for $m$ odd} \\
 V(3\lambda_1)[m-1] & \text{for $m$ even} 
  \end{cases}\, .
\]
Next we turn to the relations.
We first compute the cohomology of the complex $\gr^2 \GC_{(g),1}^\tp$.
There are four types of graphs contributing, and the complex is schematically depicted as 
\[
\begin{tikzcd}  
  \begin{tikzpicture}
    \node[int,label={$\alpha\beta$}] (v) at (0,0) {};
    \node[int,label={$\gamma\delta$}] (w) at (.7,0) {};
    \draw(v) edge (w);
  \end{tikzpicture}\ar{r}
  & 
  \begin{tikzpicture}
    \node[int,label=-90:{$\alpha\beta$}] (v) at (0,0) {};
    \draw (v) edge [loop] (v);
  \end{tikzpicture}  
  \\
  \ar{u}\ar{r}
\begin{tikzpicture}
  \node[int,label={$\alpha\beta \gamma\delta$}] (v) at (0,0) {};
\end{tikzpicture}
& 
\ar{u}
\begin{tikzpicture}
  \node[int,label={$\alpha\beta$}] (v) at (0,0) {};
  \node[int] (w) at (.7,0) {};
  \draw(v) edge (w) (w) edge[loop] (w);
\end{tikzpicture}
\end{tikzcd}
\]
The right-hand column forms an acyclic subcomplex and can be omitted -- cf. also the proof of part (i) of Theorem \ref{thm:GC comparison} in section \ref{sec:part i proof}.
The left-hand column is, up to degree shifts, given by the cocommutative coproduct 
\[
  \begin{cases}
  \wedge^4 V(\lambda_1) \to S^2(\wedge^2 V(\lambda_1)) & \text{for $m$ odd} \\
  S^4 V(\lambda_1) \to S^2(S^2 V(\lambda_1)) & \text{for $m$ even} 
  \end{cases},
\]
that is an injective map.
We record the decompositions
\begin{align*}
  \wedge^4 V(\lambda_1) &= V(0) \oplus V(\lambda_2) \oplus V(\lambda_4) \\
  S^2(\wedge^2 V(\lambda_1)) &= V(0)^2 \oplus V(\lambda_2)^2 \oplus V(\lambda_4) \oplus V(2\lambda_2)
\end{align*}
for odd $m$ and 
\begin{align*}
  S^4 V(\lambda_1) &= V(0) \oplus V(2\lambda_1)\oplus V(4\lambda_1) \\
  S^2(S^2 V(\lambda_1)) &= V(0)^2\oplus V(2\lambda_1)^2 \oplus V(2\lambda_2) \oplus V(4\lambda_1).  
\end{align*}
for even $m$.
Hence we find that, as $\OSp_g$-representations
\begin{equation}\label{equ:gr2 HGC comp}
  \gr^2 H(\GC_{(g),1}^\tp)[2-2m]
  \cong 
  \begin{cases}
    V(0) \oplus V(\lambda_2) \oplus V(2\lambda_2)  & \text{for $m$ odd} \\
    V(0)\oplus V(2\lambda_1) \oplus V(2\lambda_2)
    & \text{for $m$ even} 
    \end{cases}.
\end{equation}
We compare this to the exterior square of the generators, which are up to degree shifts obtained by the decompositions (for $g\geq 6$)
\begin{equation}\label{equ:generators squared}
  \begin{cases}
    \wedge^2V(\lambda_3) = V(0) \oplus V(\lambda_2) \oplus V(\lambda_4) \oplus V(\lambda_6)  \oplus V(2\lambda_2)  \oplus V(\lambda_2+\lambda_4)  & \text{for $m$ odd} \\
    S^2V(3\lambda_1)
    = V(0) \oplus V(2\lambda_1) \oplus V(2\lambda_2)\oplus V(4\lambda_1)
    \oplus V(2\lambda_1+2\lambda_2)\oplus V(6\lambda_1)
    & \text{for $m$ even} 
   \end{cases}.
\end{equation}
We hence find that the spaces of relations are, as $\OSp_{g}$-representations 
\[
R_{(g),1}^\tp[2-2m] \cong 
\begin{cases}
  V(0) \oplus V(\lambda_4) \oplus V(\lambda_6)  \oplus V(\lambda_2+\lambda_4)
    & \text{for $m$ odd} \\
     V(4\lambda_1)
    \oplus V(2\lambda_1+2\lambda_2)\oplus V(6\lambda_1)
  & \text{for $m$ even} 
\end{cases}
\]

{\bf Case $\GC_{(g),1}$:}
One proceeds anologously, except that one omits all graphs with tadpoles. We suppose that $g\geq 6$ for simplicity.
Concretely, the space of generators is 
\[
  V_{(g),1} := \gr^1 H(\GC_{(g),1})
  =
  \begin{cases}
 V(\lambda_1)[m-1] \oplus V(\lambda_3)[m-1] & \text{for $m$ odd} \\
 V(\lambda_1)[m-1] \oplus V(3\lambda_1)[m-1] & \text{for $m$ even} 
  \end{cases}\, .
\]
By Theorem \ref{thm:GC comparison} we have that $\gr^2 H(\GC_{(g),1})=\gr^2 H(\GC_{(g),1}^\tp)$, so we can re-use the computations from above for that space.
The exterior square of the generators decomposes as 
\[
  \begin{cases}
    \wedge^2(V(\lambda_1)\oplus V(\lambda_3))
     = V(\lambda_2)^2 \oplus V(\lambda_4) \oplus V(\lambda_3+\lambda_1) \oplus 
     \wedge^2( V(\lambda_3))  & \text{for $m$ odd} \\
    S^2(V(\lambda_1)\oplus V(3\lambda_1))
    = V(2\lambda_1)^2
     \oplus V(2\lambda_1+\lambda_2)
      \oplus V(4\lambda_1)
    \oplus 
    S^2V(3\lambda_1)
    & \text{for $m$ even} 
   \end{cases}.
\]
Hence we find that the space of generators is, as $\OSp_g$-representation 
\[
R_{(g),1}[2-2m] \cong R_{(g),1}^\tp \oplus
\begin{cases}
  V(\lambda_2)^2 \oplus V(\lambda_4) \oplus V(\lambda_3+\lambda_1)
    & \text{for $m$ odd} \\
    V(2\lambda_1)^2
    \oplus V(2\lambda_1+\lambda_2)
     \oplus V(4\lambda_1)
  & \text{for $m$ even} 
\end{cases}\, .
\]

{\bf Case $\GC_{(g)}^{\ex}$:}
The generating space 
\[
V_{(g)} := \gr^1 H(\GC_{(g)}^{\ex}).  
\]
can again be copied from section \ref{sec:weight 1 comparison}. We have
\[
  V_{(g)}[1-m]\cong V_{(g),1}^\tp[-m]\cong
\begin{cases}
  V(\lambda_3)
    & \text{for $m$ odd} \\
    V(3\lambda_1)
  & \text{for $m$ even} 
\end{cases}.
\]
Next we consider the complex $\gr^2 \GC_{(g)}^{\ex}=\gr^2 \GC_{(g)}$.
Graphically, it has the following terms 
\[
\begin{tikzcd}  
  \begin{tikzpicture}
    \node[int,label={$\alpha\beta \omega$}] (v) at (0,0) {};
  \end{tikzpicture}
  \ar{r}
  & 
  \begin{tikzpicture}
    \node[int,label={$\alpha\beta$}] (v) at (0,0) {};
    \node[int,label={$\gamma\delta$}] (w) at (.7,0) {};
    \draw(v) edge (w);
  \end{tikzpicture}
  \\
  &
  \ar{u}
\begin{tikzpicture}
  \node[int,label={$\alpha\beta \gamma\delta$}] (v) at (0,0) {};
\end{tikzpicture}
\end{tikzcd}
\]
The right-hand column is the complex computing $\gr^2 H(\GC_{(g),1})$, so we ca re-use the computation from above.
The graphs in the upper left have two decorations in $H_m(W_g)$ and one ($\omega$) in $H_{2m}(W_g)$. Representation-wise, this contributes 
\[
  \begin{cases}
    V(0)\oplus V(\lambda_2)
      & \text{for $m$ odd} \\
    V(0) \oplus  V(2\lambda_1)
    & \text{for $m$ even} 
  \end{cases}.
\]

The upper horizontal arrow replaces $\omega$ by an edge towards a new vertex decorated by the reduce diagonal $\Delta_1$.
One checks by a small computation that this map induces an injection, and so (using \eqref{equ:gr2 HGC comp}) we find that 
\[
\gr^2 H(\GC_{(g)}^{\ex}) \cong V(2\lambda_2)[2m-2],  
\]
and the formula is valid for both even and odd $m$.
Using \eqref{equ:generators squared} for the exterior square of the generators we determine that the relations are 
\[
R_{(g)} \cong 
\begin{cases}
  V(0) \oplus V(\lambda_2) \oplus V(\lambda_4) \oplus V(\lambda_6)   \oplus V(\lambda_2+\lambda_4)  & \text{for $m$ odd} \\
  V(0) \oplus V(2\lambda_1) \oplus V(4\lambda_1)
  \oplus V(2\lambda_1+2\lambda_2)\oplus V(6\lambda_1)
  & \text{for $m$ even} 
\end{cases}
\]

\subsection{Proofs of statements of the introduction}
We finally turn to the remaining statements of the introduction. First, Theorem \ref{thm:main t} is one case of Theorem \ref{thm:presentations}, if combined with the explicit computation of the relations in this subsection.
Next, the zigzag of Theorem \ref{thm:main HCE} can be taken to be taken to be 
\[
C_{CE}(\GCex_{(g)}) \leftarrow 
\tru^{[m]} C_{CE}(\GCex_{(g)})
\to 
H^{[m]}_{CE}(\GCex_{(g)})
\leftarrow A_{(g)}.
\]
This induces isomorphisms on the cohomolgy in weight $W$ cohomology if $g\geq 3W\geq 6$ by Theorem \ref{thm:presentations}, so that Theorem \ref{thm:main HCE} is shown.
Finally, Corollary \ref{cor:Hain} is just a less precise rewording of statement (vi) of Proposition \ref{prop:cohom conc follows koszul 2} applied to $\fg=\GCex_{(g)}$ in the case $m=1$.









\end{document}